\documentclass[12pt,british,a4paper,reqno]{amsart}
\usepackage[T1]{fontenc}
	\usepackage[latin9]{inputenc}
	\usepackage[british]{babel}

	\usepackage{amsmath}
	\usepackage{amssymb}
	\usepackage{amsthm}
	\usepackage{braket}
\usepackage[dvipsnames]{xcolor}
\usepackage[shortlabels]{enumitem}
	\usepackage{fancyhdr}
	\usepackage{geometry}
	\usepackage{graphicx}
	\usepackage{hyperref}
 \usepackage{cleveref}
		\usepackage{scalerel}\usepackage{bookmark}
\usepackage{float}
	\usepackage{indentfirst}
 \usepackage{bbding}

	\usepackage{mathrsfs}
	\usepackage{mathtools}
	\usepackage{proof}
	\usepackage{qtree}
	\usepackage{setspace}
	\usepackage{tensor}
	\usepackage{tikz}
	\usepackage{tikz-cd}
	\usepackage{tabstackengine}
\usepackage[singlelinecheck=false]{caption}
	\setstackgap{L}{.75\baselineskip}
\usetikzlibrary{chains}

\usepackage{caption}
    \usepackage{subcaption}
	\bookmarksetup{numbered,open}

	\renewcommand{\geq}{\geqslant}
	\renewcommand{\leq}{\leqslant}
	\renewcommand{\phi}{\varphi}

	\providecommand{\corollaryname}{Corollary}
	\providecommand{\definitionname}{Definition}
	\providecommand{\examplename}{Example}
	\providecommand{\lemmaname}{Lemma}
	\providecommand{\propositionname}{Proposition}
	\providecommand{\remarkname}{Remark}
	\providecommand{\theoremname}{Theorem}
	\providecommand{\setupname}{Setup}
	\providecommand{\conjecturename}{Conjecture}
	\providecommand{\questionname}{Question}
	\providecommand{\objectivename}{Objective}
	\providecommand{\aimname}{Aim}

	\theoremstyle{plain}
		\newtheorem{thm}{\protect\theoremname}[section] 
		\newtheorem{prop}[thm]{\protect\propositionname}
		\newtheorem{lem}[thm]{\protect\lemmaname}
		\newtheorem{cor}[thm]{\protect\corollaryname}

	\theoremstyle{definition}
		\newtheorem{defn}[thm]{\protect\definitionname}
		\newtheorem{example}[thm]{\protect\examplename}
		\newtheorem{setup}[thm]{\protect\setupname}

	\theoremstyle{remark}
		\newtheorem{rem}[thm]{\protect\remarkname}
		
	\numberwithin{figure}{section}
	\numberwithin{equation}{section}

	\newenvironment{acknowledgements}{
		\begin{abstract}} {\end{abstract}}

	\usetikzlibrary{matrix,arrows,decorations.pathmorphing,positioning,decorations.pathreplacing}
	\tikzset{commutative diagrams/.cd, 
		mysymbol/.style = {start anchor=center, end anchor = center, draw = none}}

	\let\amph=& 

	\newcommand{\BA}{\mathbb{A}}
	\newcommand{\BC}{\mathbb{C}}
	\newcommand{\BD}{\mathbb{D}}
	\newcommand{\BE}{\mathbb{E}}
	\newcommand{\BF}{\mathbb{F}}
	\newcommand{\BH}{\mathbb{H}}

	\newcommand{\BR}{\mathbb{R}}
	
	\newcommand{\BZ}{\mathbb{Z}}

		\newcommand{\Aut}{\operatorname{Aut}\nolimits}
		\newcommand{\rad}{\operatorname{rad}\nolimits}

  \newcommand{\zembyk}{\scaleobj{0.7}{\mbox{\ScissorRight\hspace{-0.5mm}}}}
		
	\makeatletter
\DeclareRobustCommand{\rvdots}{%
  \vbox{
    \baselineskip4\p@\lineskiplimit\z@
    \kern-\p@
    \hbox{.}\hbox{.}\hbox{.}
  }}
\makeatother


 \usepackage{csquotes}

\definecolor{CSBRed}{RGB}{190, 15, 52}
\definecolor{kaylapink}{RGB}{219, 48, 122}

\usepackage{fullpage}

\title{A geometric model for semilinear locally gentle algebras}

    \author[Banaian]{Esther Banaian}
        \address{
		Department of Mathematics\\
		Aarhus University\\
		Ny Munkegade 118\\
		8000 Aarhus C\\
		Denmark}
        \email{banaian@math.au.dk}
    \author[Bennett-Tennenhaus]{Raphael Bennett-Tennenhaus}
        \address{
		Department of Mathematics\\
		Aarhus University\\
		Ny Munkegade 118\\
		8000 Aarhus C\\
		Denmark}
        \email{raphaelbennetttennenhaus@gmail.com}
            \author[Jacobsen]{Karin M. Jacobsen}
        \address{
		Department of Mathematics\\
		Aarhus University\\
		Ny Munkegade 118\\
		8000 Aarhus C\\
		Denmark}
        \email{karin.jacobsen@math.au.dk}
            \author[Wright]{Kayla Wright}
        \address{
		School of Mathematics\\
		University of Minnesota\\
		504 Vincent Hall\\
		206 Church Street\\
            Minneapolis, MN 55455\\
           United States of America}
        \email{kaylaw@umn.edu}

\date{}

\begin{document}

\maketitle

\begin{abstract}
We consider certain generalizations of gentle algebras that we call semilinear locally gentle algebras. 
These rings are examples of semilinear clannish algebras as introduced by the second author and Crawley-Boevey \cite{BenTenCraBoe}. 
We generalise the notion of a nodal algebra from work of Burban and Drozd \cite{Burban-Drozd-derived-nodal-algebras} and prove that semilinear gentle algebras are nodal by adapting a theorem of Zembyk \cite{Zembyk-Skewed-Gentle-A}.  We also provide a geometric realization of Zembyk's proof, which involves cutting the surface into simpler pieces in order to endow our locally gentle algebra with a semilinear structure. We then consider this surface glued back together, with the seams in place, and use it to give a geometric model for the finite-dimensional modules over the semilinear locally gentle algebra.


\end{abstract}

\tableofcontents
\section{Introduction}


\subsection{Background} We look at rings which we call \emph{semilinear locally gentle algebras}. 
To provide context we recall and relate gentle algebras, nodal algebras and surface algebras.  

\subsubsection*{Gentle algebras. } Assem and Skowro\'{n}ski \cite{assem-skowronski-gentle} defined what are now known as 
\emph{gentle algebras} as 
iterated tilted algebras of type $\mathbb{A}$ and $\tilde{\mathbb{A}}$.  
These comprise an important subclass of the class of \emph{string algebras} defined by Butler and Ringel  \cite{Butler-Ringel-string-algebras}, and hence both their 
modules 
and their Auslander--Reiten sequences have an explicit description. 
Gentle algebras also benefit from well-behaved homological properties, and have been studied intensely from a range of perspectives; 
see \cite{arnesen-laking-pauksztello,bekkert-merklen-derived-cat-gentle,BenTenthesis,canakci-pauksztello-schroll,Dequene-Jordan-recoverability,geiss-reiten-iwanaga-gorenstein,kaclk-singularity,Marczinzik-Rubey-Stump,STTYW,Shen-Wu-magnitude}. 

Bessenrodt and Holm \cite{Bessenrodt-Holm-weighted} introduced \emph{locally gentle algebras} using a direct generalisation of the gentle algebras from \cite{assem-skowronski-gentle} to include infinite-dimensional algebras such as $K[x,y]/(xy)$, the coordinate ring of a nodal point. 
Locally gentle algebras define a subclass of the string algebras considered by Crawley-Boevey \cite{craboe-inf-dimstring}, who classified finitely generated modules over such algebras.  
Other generalisations of  gentle algebras provide insight outside the focus of finite-dimensional algebras; see \cite{bessholmskewgentle,BenTen2023,ricke}.



In more recent work of the second author and Crawley-Boevey \cite{BenTenCraBoe},  \emph{semilinear clannish algebras} were introduced. 
Such rings simultaneously generalise locally gentle algebras and the \emph{clannish algebras} introduced by Crawley-Boevey \cite{Crawley-Boevey-clans}. 
For semilinear clannish algebras one uses the \emph{semilinear path algebra} in which each arrow comes equipped with an element of the automorphism group of the underlying division ring so that the quiver representations that correspond to modules must be made up of semilinear maps. 

In this article, we look at a particular class of semilinear clannish algebras which we call \emph{semilinear locally gentle algebras}. These rings generalise locally gentle algebras, but include other examples, such as $K[x,y;\sigma,\sigma^{-1}]/(xy)$ for an automorphism $\sigma$ of $K$. 
When $\sigma$ is the Frobenius automorphism of a perfect field $K$ of positive characteristic,  finite-dimensional $K[x,y;\sigma,\sigma^{-1}]/(xy)$-modules correspond to Dieudonn\'e modules.

\subsubsection*{Nodal algebras. }
Burban and Drozd introduced \emph{nodal algebras} in \cite{Burban-Drozd-derived-nodal-algebras},  inspired by work of Drozd and Greuel on classifications of vector bundles on projective curves \cite{drozd-greuel-projective-curves}. 
The definition of a nodal algebra requires the existence of an embedding into a hereditary algebra  satisfying compatibility conditions. 
These conditions arose in work of Drozd \cite{drozd-purely-noetherian}. Such conditions characterise the representation-tame \emph{purely noetherian complete} algebras, and it is seen in \cite{Burban-Drozd-derived-nodal-algebras} that said conditions are particularly well-suited to study the derived category of a nodal algebra using that of its connected hereditary algebra.
Burban and Drozd have since generalised this concept to the notion of a \emph{nodal ring} \cite{Burban-Drozd-derived-nodal-rings}. 

The representation theory of nodal algebras of classical and extended Dynkin types have been studied extensively. 
Type $\BA$ were studied by Drozd and Zembyk in \cite{nodal-type-A}. 
These authors also looked at type $\BD$  in \cite{nodal-type-D}. 
Type $\BE$ nodal algebras were studied by Drozd, Golovashchuk, and Zembyk \cite{nodal-type-E}. 
Of particular interest in this article are nodal algebras of classical type $\BA$. 
Zembyk \cite{Zembyk-Skewed-Gentle-A} showed that gentle 
algebras are all nodal algebras of type $\BA$.  
The arguments used in \cite{Zembyk-Skewed-Gentle-A} define a focal point in this article.

\subsubsection*{Surface  algebras. }
Assem, Br{\"u}stle, Charbonneau--Jodoin and Plamondon \cite{Assem-Brustle-Charbonneau-Jodoin-Plamondon-gentle-arising-from-surface-triangulations} studied the \textit{Jacobian algebra} of a \textit{quiver with potential} in the sense of Derksen, Weyman and Zelevinsky \cite{derk-wey-zel-1-mutations}, defined by a triangulation of a non-punctured marked surface. The main result in \cite{Assem-Brustle-Charbonneau-Jodoin-Plamondon-gentle-arising-from-surface-triangulations} says that the cluster-tilted algebras of type $\mathbb{A}$ or $\tilde{\mathbb{A}}$ is exactly the class of Jacobian algebras arising from a triangulation of a disc or an annulus. 
Moreover finite-dimensional modules
for such algebras may be modelled using arcs on this surface, and morphisms between them also have a geometric interpretation. 
Rings defined by topological spaces, and the interplay between representations and geometry, have played an important role in algebra; see for example \cite{Amiot-Plamondon-Schroll-derived-invariant,Baur-Laking,Baur-Schroll-Gentle-Extensions,bennetttennenhaus2023semilinear,Bial-Erdmann-Hajduk-Skowronski-Yamagata,Chang-Schroll-exceptional-sequences,Fu-Geng-Liu-Zhou-support-tau-tilting,GarverPatriasThomas2023,geiss2022schemes,HJS-higher-homological-gentle,Gelfand-Gelfand-Retakh-I, labardini2023gentle,LabFra-quivers-iv,lekili2023homological, opper2018geometric}.

The observations in \cite{Assem-Brustle-Charbonneau-Jodoin-Plamondon-gentle-arising-from-surface-triangulations} have since been generalised to the setting of gentle algebras. 
Firstly, Baur and Coelho-Simoes \cite{Baur-Coelho-Simoes-geometric-model-module-cat-gentle} showed that every gentle algebra arises as a so-called \emph{tiling algebra}, and also described how the surface and tiling together provide a geometric model for the category of finite-dimensional modules. 
Palu, Pilaud and Plamondon  
\cite{Palu-Pilaud-Plamondon-non-kissing-non-crossing} generalised this observation to locally gentle algebras  using the language of \emph{surface dissections}. 
The language used in \cite{Palu-Pilaud-Plamondon-non-kissing-non-crossing} is central to some of the main results in this article. 


\subsection{Main results}

The locally gentle algebras from  \cite{Bessenrodt-Holm-weighted} are defined using a \emph{locally gentle} pair  $(Q,Z)$, so $Q$ is a quiver and $Z$ is a set of paths in $Q$ of length $2$. 
The proof of the main result from \cite{Zembyk-Skewed-Gentle-A} involves the construction of a new quiver $Q^{\zembyk}(Z)$ from $(Q,Z)$.

\begin{figure}[H]
	\begin{subfigure}[c]{0.35\textwidth}   
	\begin{tikzpicture}
\node (1'') at (2,0.5) {1};
	\node (2'') at (2.75,1.1) {$2$};
	\node (3'') at (2.75,-0.1) {$3$};
	
	\node (5'') at (3.75,1.1) {$5$};
        \node (4'') at (3.75,-0.1) {$4$};
	\node (6'') at (4.5,0.5) {6};

	\draw[->] (1'') to (2'');
	\draw[->] (2'') to node[near start, name=a] {} 
 node[near end, name=a'] {} (3'');
	\draw[->] (3'') to (1'');
	\draw[->] (5'') to node[near start, name=b] {} 
 node[near end, name=b'] {} (2'');
	\draw[->] (4'') to node[near start, name=c] {} 
 node[near end, name=c'] {} (5'');
	\draw[->] (3'') to node[near start, name=d] {} 
 node[near end, name=d'] {} (4'');
	\draw[->] (5'') to (6'');

 \draw[dashed, bend left] (b') to (a);
  \draw[dashed, bend left] (c') to (b);
   \draw[dashed, bend left] (d') to (c);
    \draw[dashed, bend left] (a') to (d);
  
\end{tikzpicture}
 	\end{subfigure}
	\begin{subfigure}[c]{0.24\textwidth}   
	\begin{tikzpicture}

	\node (1) at (4.75,0.5) {1};
	\node (2) at (5.5,1) {$2'$};
	\node (3) at (5.5,0) {$3'$};
	\node (3') at (6.1,0) {$3''$};
	\node (2') at (6.1,1) {$2''$};
	\node (5) at (7,1) {$5'$};
        \node (4') at (7,0) {$4'$};
         \node (4) at (7.6,0) {$4''$};
	\node (5') at (7.6,1) {$5''$};
	\node (6) at (8.35,0.5) {6};

	\draw[->] (1) to (2);
	\draw[->] (2) to (3);
	\draw[->] (3) to (1);
	\draw[->] (5) to (2');
	\draw[->] (4) to (5');
	\draw[->] (3') to (4');
	\draw[->] (5') to (6);

\end{tikzpicture}
  	\end{subfigure}

 	\caption{
Quiver $Q$ with zero-relations $Z$ depicted by dashes (left), 
  and the corresponding quiver 
$ Q^{\protect\zembyk}(Z)$, inspired by   \cite[p. 649, Theorem]{Zembyk-Skewed-Gentle-A} (right).}
    \label{fig-intro-quivers}
\end{figure}

\Cref{thm-main-result-1}   geometrically realises the construction of $Q^{\zembyk}(Z)$ from the pair $(Q,Z)$ using surface dissections in the sense of \cite{Palu-Pilaud-Plamondon-non-kissing-non-crossing}. 
More specifically, we start with a surface dissection $(S,D)$ corresponding to a locally gentle pair $(Q,Z)$. 
We then choose a subset $R^{*}$ of the dual dissection $D^{*}$ and define what we call the \emph{split}  of $S$, which is a collection of disjoint surface dissections found by slicing up $(S,D)$ is a particular way. 
The subset $R^{*}$ of the arcs in $D^{*}$ is  determined by the \emph{relational} vertices in our quiver meaning those arising as $h(a)=t(b)$ for some zero-relation  $ba\in Z$.

\begin{figure}[H]

	\begin{subfigure}[c]{0.35\textwidth}   
	\begin{tikzpicture}

 \draw[rounded corners] (6,0.6)--(6,2.1)--(10,2.1)--(10,0.1)--(6,0.1)--(6,0.6);

        \begin{scope}[blue]
        \node[]  (a') at (6.05,2.05) {$\bullet$};
        \node[] (b') at (6.5,0.1) {$\bullet$};
        \node[] (c') at (7.5,1.35) {$\bullet$};
        \node[] (d') at (8.5,2.1) {$\bullet$};
        \node[] (e') at (9,0.1) {$\bullet$};
        \node[] (f') at (10,1.6) {$\bullet$};

        \draw (a'.center) to[bend right] (c'.center);
        \draw (d'.center) to[bend left] (c'.center);
        \draw (b'.center) to[bend left] (c'.center);
        \draw (b'.center) to[bend left] (e'.center);
        \draw (d'.center) to[bend left] (e'.center);
        \draw (f'.center) to[bend left] (e'.center);

        \end{scope}
        \begin{scope}[red, densely dashed]
        \node (i') at (7,2.1) {$\bullet$};
        \node (ii') at (6,0.6) {$\bullet$};
        \node (iii') at (8.5,1.6) {$\bullet$};
        \node (iv') at (7.75,0.1) {$\bullet$};
        \node (v') at (9.5,2.1) {$\bullet$};
        
        \draw (i'.center) to[bend right] (iii'.center);
        \draw (ii'.center) to[bend right] (iii'.center);
        \draw (iv'.center) to[bend right] (iii'.center);
        \draw (v'.center) to[bend right] (iii'.center);
        \end{scope}
    \end{tikzpicture}
 	\end{subfigure}
  \quad  \quad  \quad
	\begin{subfigure}[c]{0.24\textwidth}   
	\begin{tikzpicture}[inner sep=1, outer sep = 0]
        \draw[rounded corners] (-0.25,0.8)--(-0.25,2.3)--(0.75,2.3);
        \begin{scope}[black]
        \node (i) at (0.75,2.3) {};
        \node (ii) at (-0.25,0.8) {};
        \node (iii) at (2.25,1.6) {};
        
        \draw (i.center) to[bend right] (1.85,1.6);
        \draw (1.85,1.6) to (iii.center);
        \draw (iii.center) to[bend left] (0.25,0.75);
        \draw (0.25,0.75) to (ii.center);
        \end{scope}
        
        \begin{scope}[OliveGreen]
        \node[]  (a) at (-0.2,2.25) {$\bullet$};
        \node[] (b) at (0.25,0.75) {$\bullet$};
        \node[] (c) at (1.25,1.55) {$\bullet$};
        \node[] (d) at (1.85,1.6) {$\bullet$};

        \draw (a.center) to[bend right] (c.center);
        \draw (d.center) to[bend left] (c.center);
        \draw (b.center) to[bend left] (c.center);
        \end{scope}

            \draw[rounded corners] (0,0.5)--(0,0)--(1.75,0);
        \begin{scope}[black]
        \node[] (x) at (0,0.5) {};
        \node[] (y) at (2.5,1.5) {};
        \node[] (z) at (1.75,0) {};
        \draw (x.center) to (0.55,0.465);
        \draw (0.55,0.465) to[bend right] (y.center);
        \draw (z.center) to (2.125,0.3);
        \draw (2.125,0.3) to[bend right] (y.center);
        \end{scope}       

         \begin{scope}[OliveGreen]
        \node[] (l) at (0.5,0) {$\bullet$};
        \node[] (m) at (0.55,0.465) {$\bullet$};
        \node[] (n) at (2.125,0.3) {$\bullet$};
        \draw (l.center) to[bend left] (m.center);
        \draw (l.center) to[bend left] (n.center);
        \end{scope}

             \draw[rounded corners] (1.5,2.3)--(3,2.3)--(3.6,2.3);
         \begin{scope}[black]
        \node (j) at (1.5,2.3) {};
        \node (jjj) at (3,1.8) {};
        \node (u) at (3.6,2.3) {};
        
        \draw (j.center) to[bend right =20] (2.25,1.8);
        \draw (2.25,1.8) to[bend right=10] (jjj.center);
        \draw (u.center) to[bend right=20] (3.25,2.075);
        \draw (3.25,2.075) to[bend right=10] (jjj.center);
        \end{scope}
        
        \begin{scope}[OliveGreen]
        \node[] (g) at (2.25,1.8) {$\bullet$};
        \node[] (h) at (3,2.3) {$\bullet$};
        \node[] (k) at (3.25,2.075) {$\bullet$};

        \draw (h.center) to[bend left] (g.center);
        \draw (h.center) to[bend left] (k.center);       
        \end{scope}

            \draw[rounded corners] (4,2)--(4.5,2)--(4.5,0)--(2.25,0);

\begin{scope}[black]
        \node (ooo) at (3,1.5) {};
        \node (oooo) at (2.25,0) {};
        \node (ooooo) at (4,2) {};
        \node (oooooo) at (4.5,0.5) {};
        
        \draw (oooo.center) to[bend right=10] (2.7,0.4);
         \draw (2.7,0.4) to[bend right=20] (ooo.center);
        \draw (ooooo.center) to[bend right=20] (3.25,1.75);
        \draw (3.25,1.75) to[bend right=10] (ooo.center);
        \end{scope}

        \begin{scope}[OliveGreen]
        \node[] (bb) at (2.7,0.4) {$\bullet$};
        \node[] (dd) at (3.25,1.75) {$\bullet$};
        \node[] (ee) at (3.5,0) {$\bullet$};
        \node[] (ff) at (4.5,1.5) {$\bullet$};
           
        \draw (bb.center) to[bend left] (ee.center);
        \draw (dd.center) to[bend left] (ee.center);
        \draw (ff.center) to[bend left] (ee.center);
        \end{scope}

\end{tikzpicture}
 	\end{subfigure}
 	\caption{Dissection $(S,V,D)$ (blue) with a subset $R^{*}$ (red) of its dual $D^{*}$;   and its split (green) corresponding to   \Cref{fig-intro-quivers}. }
    \label{fig-intro-surfaces}
\end{figure}
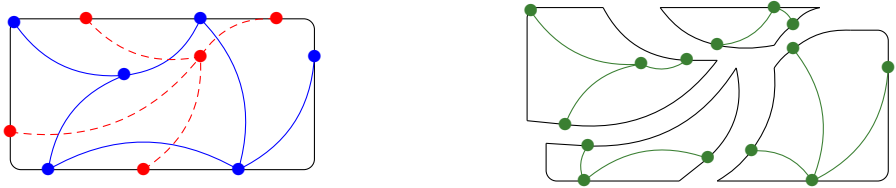



\begin{thm}
\label{thm-main-result-1}
    (see \Cref{thm:SplitVsZembyk}.) 
    Let $\overline{Q}=(Q,Z)$ be a locally gentle pair with surface dissection $(S_{\overline{Q}}, V_{\overline{Q}},D_{\overline{Q}})$.  The quivers associated to the connected components of the split of $(S_{\overline{Q}}, V_{\overline{Q}},D_{\overline{Q}})$ along $R^*$  bijectively  correspond to the connected components of $Q^{\zembyk}(Z)$.
\end{thm}



The focus of this article is on what we call \textit{semilinear locally gentle algebras}. 
Any such algebra has the form $K_{\boldsymbol{\sigma}}Q/\langle Z\rangle$ where: $(Q,Z)$ is a locally gentle pair; $\boldsymbol{\sigma}$ is a collection  of automorphisms $\sigma_{a}$ of the division ring $K$, one for each arrow $a$ of $Q$; and $K_{\boldsymbol{\sigma}}Q$ is the corresponding \emph{semilinear path algebra}. 
Hence semilinear locally gentle algebras  define a subclass of semilinear clannish algebras from \cite{BenTenCraBoe}.  

We refer to the semilinear locally gentle algebras that are finite-dimensional over $K$ as \textit{semilinear gentle algebras}. 
By  adapting the notion of a nodal algebra to our more general setting, we generalise the argument used in  the proof of \cite[p. 649, Theorem]{Zembyk-Skewed-Gentle-A}.

\begin{thm}
\label{thm-main-result-2}
    (see \Cref{thm-semilinear-fd-gentle-implies-nodal}.) 
     Any finite-dimensional semilinear gentle algebra $\Lambda=K_{\boldsymbol{\sigma}}Q/\langle Z\rangle  $ is nodal, connected with $\Gamma=K_{\boldsymbol{\sigma}}Q^{\zembyk}(Z)$. 
\end{thm}


Relying on a classification from the second author and Crawley-Boevey in \cite{BenTenCraBoe}, our model gives a geometric interpretation of the finite-dimensional  modules of the semilinear locally gentle algebra. Namely, we first connect two pieces of literature to provide a geometric model for indecomposable modules for a locally gentle algebra. In particular, in \cite{Baur-Coelho-Simoes-geometric-model-module-cat-gentle}, the authors give a geometric model for the module category for gentle algebras. Using the  pair of surface dissections from \cite{Palu-Pilaud-Plamondon-non-kissing-non-crossing}, we are able to take the proof ideas in \cite{Baur-Coelho-Simoes-geometric-model-module-cat-gentle} to give a geometric realization of indecomposable modules for locally gentle algebras. 

\begin{thm}
\label{thm-main-result-3}
    (see \Cref{cor:BijArcsModules}.)
Given a semilinear locally gentle algebra $\Lambda=K_{\boldsymbol{\sigma}}Q/\langle Z\rangle  $, there exists a surface with a pair of dual dissections such that there is a correspondence between isomorphism classes of indecomposable finite-dimensional modules $\Lambda$ and equivalence classes of permissible arcs or closed curves.
\end{thm}

The finite-dimensional indecomposable $\Lambda$-modules from \Cref{thm-main-result-3} have the form $
M(C,V)=M(C)\otimes_{\Lambda(C)} V$ where: $C$ is a \emph{string} or a \emph{band}, a piece of combinatorial data chosen with respect to $(Q,Z)$;  $\Lambda(C)$ is a $K$-ring, either $K$ or of the form $K[t,t^{-1};\pi]$; $M(C)$ is an explicitly defined  $\Lambda$-$\Lambda(C)$ bimodule; and $V$ is an arbitrary finite-dimensional indecomposable $\Lambda(C)$-module. 
In \Cref{thm-main-result-4} we use the phrase \emph{semilinear structure of} $M(C,V)$ to refer to the right action of $\Lambda(C)$ on $M(C)$.

For the locally gentle pair $(Q,Z)$ the collection of components $(S_i,V_{D_i},D_i)$ found by the split of the  surface associated to $(Q,Z)$ have an interesting geometric feature: we can associate each arrow from the original quiver uniquely with a face in the union of the faces of these components. 
Combining this bijection with \Cref{thm-main-result-3}, we note how  the automorphism face labellings capture the semilinearity of representations of quivers that correspond to modules over semilinear locally gentle algebras. 

\begin{thm}
\label{thm-main-result-4}
    (see \Cref{thm:PiMatchesArcSemilinearity}.)
Given a finite-dimensional module $M$ over a semilinear locally gentle algebra $\Lambda=K_{\boldsymbol{\sigma}}Q/\langle Z\rangle $, the semilinear structure of $M$ can be completely described in terms of a \emph{labeled tiling} of the surface associated to $\Lambda$. 
\end{thm}



\subsection{Structure of the paper}

Our paper is structured as follows: In \S\ref{sec:background} we give background on locally gentle quivers, as well as Zembyk's algorithm for decomposition of quivers. In \S\ref{sec:SurfaceDissection} we define the underlying surfaces corresponding to a gentle quiver, as well as related geometric and topological notions.  In \S\ref{sec:geozembyk}, we provide a geometric description of Zembyk's algorithm. In \S\ref{sec-semilinearity} we recall nodal algebras, show that semilinear gentle algebras are nodal, and give background on semilinear locally gentle algebras needed for the geometric model for semilinear representations in \S\ref{sec-geometric-model}.

\section{Locally Gentle Bound Quivers and Zembyk's Algorithm}\label{sec:background}

In this section, we review the notion of bound quivers and describe which sets of relations are so-called locally gentle. We then discuss work of Zembyk \cite{Zembyk-Skewed-Gentle-A} that describes a way to cut up certain bound quivers into quivers without relations. 
We begin by setting up some notation, conventions and terminology, most of which is standard.

\begin{setup}
    Let $Q$ be a quiver with a set $Q_{0}$ of vertices, a set $Q_{1}$ of arrows, and functions $h,t\colon Q_{1}\to Q_{0}$ assigning to each arrow $  a $ its head $h(  a )$ and tail $t(  a )$ i.e. $t(  a ) \xrightarrow{  a } h(  a )$. 
    A \emph{path} in $Q$ means either a \emph{trivial path at} $v$, denoted $e_{v}$ and defined for each $v\in Q_{0}$, or a \emph{non}-\emph{trivial path} $ a_{1}\dots  a_{n}$ \emph{of length} $n\geq 1$, defined for arrows $ a_{i}$ where  $t( a_{i})=h( a_{i+1})$ for $i<n$. 
    A \emph{subpath} of a non-trivial path $p= a_{1}\dots  a_{n}$ is a path $q$ where: 
\begin{enumerate}
    \item $q=e_{v}$ is trivial where $v$ is the head or tail of some $ a_{i}$; 
    \item or $q= a_{i}\dots  a_{i+m-1}$ is non-trivial of length $m\leq n$ for some $i\geq 1$. 
\end{enumerate} 
Let $Z$ be a fixed set of \emph{quadratic zero}-\emph{relations} (so paths of  length $2$) in $Q$. 
By an \emph{admissible path} (with respect to $(Q,Z)$) we mean one with no element of $Z$ arising as a subpath. By an \emph{inadmissible path} we mean one with a subpath in $Z$. 
\end{setup}

\begin{defn}
\label{def-locally-gentle-pair}
\cite{Bessenrodt-Holm-weighted} 
We say that $(Q,Z)$ is \emph{locally gentle} given (i) and (ii) below hold. 
\begin{enumerate}[(i)]
    \item Any vertex of $Q$ is the head (respectively, tail) of at most two arrows. 
    \item Any arrow $b$ gives rise to at most one admissible (respectively, inadmissible) length $2$ path of the form $cb$, and at most one such path of the form $ba$.
\end{enumerate}
We say $(Q,Z)$ is \emph{gentle} if it is locally gentle and there are finitely many admissible paths. 
\end{defn}

\begin{rem}\label{rem:gentlevslocallygentle}
This definition of locally gentle pair does not forbid admissible paths of infinite length. 
This is the difference between gentle pairs and locally gentle pairs. 
\end{rem}

\begin{example}\label{example:gentlevslocallygentle}
 Consider the following quiver $Q$ and the relations $Z_1 = \{\beta^2,\nu\alpha\}$, $Z_2 = \{\beta \alpha, \nu \beta\}$. The pair $(Q, Z_1)$ is gentle whereas the pair $(Q, Z_2)$ is locally gentle.
 
\begin{center}
\begin{tikzpicture}[xscale=2]
	\node (1) at (1,0) {1};
	\node (2) at (2,0) {2};
	\node (3) at (3,0) {3};
	\draw[->] (1) to node[above]{$\alpha$}(2);
	\draw[->] (2) .. controls (1.7,1.5) and (2.3,1.5) .. node[pos=.75, right]{$\beta$}(2);
	\draw[->] (2) to node[above]{$\nu$}(3);
\end{tikzpicture}
\end{center}

\end{example}

\subsection{Zembyk's Construction} \label{subsec:zembykalgorithm}

In \Cref{defn-relational} we follow a construction of Zembyk \cite{Zembyk-Skewed-Gentle-A}, and at the same time we choose notation to be used later in the article. 

\begin{setup}
    Throughout \S\ref{subsec:zembykalgorithm} we assume $Q$ is a quiver and $Z$ is a set of quadratic zero-relations such that $(Q,Z)$ is locally gentle. 
\end{setup}

\begin{defn}\label{defn-relational}
Let $v$ be a vertex, and for the purposes of stating \Cref{defn-relational},  let
\[
\begin{array}{cc}
    X=(x\in Q_{1}\colon t(x)=v),
    &
    Y=(y\in Q_{1}\colon h(y)=v).
\end{array}
\]
Note that $X$ and $Y$ are possibly empty sequences of length at most $2$.  We say that $v$ is:
\begin{itemize}
    \item  a \emph{stream} if $
X=(b)$ and $Y=(a)$ where $Z\ni ba$. 
\item a \emph{tributary} if $
X=(b)$ and $Y=(a,c)$ where $bc\notin Z\ni ba $; 
\item a \emph{distributary} if $
X=(b,d)$ and $Y=(a)$ where $da\notin Z\ni ba $; and 
\item a \emph{quadtributary} if $
X=(b,d)$ and $Y=(a,c)$ where $bc,da\notin Z\ni ba,dc $. 
\end{itemize}
We say that $v$ is \emph{relational} if it is a stream, a tributary, a distributary or a quadbutary.
Note that  $v$ is relational if and only if there exists a length-$2$ path $ba\in Z$ such that $t(b)=v=h(a)$. 
The different cases for a relational vertex are drawn below. 
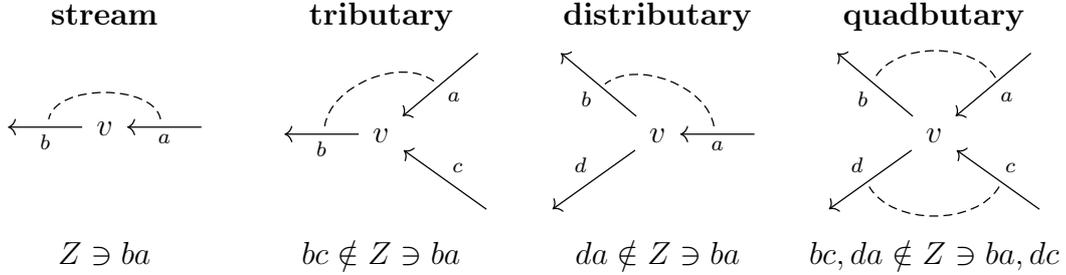
\[
\begin{array}{cccc}
\textbf{stream}

&

\textbf{tributary}

&

\textbf{distributary} 

&

\textbf{quadbutary}

\\

\begin{tikzcd}
{} & v\ar[l, "b"name=bet] & {}\ar[l,  "a"name=alp]
\ar[-, shorten <= 3pt, shorten >= 3pt, from=alp, to=bet, bend right = 80, dashed]
\end{tikzcd}

&
\begin{tikzcd}
{} & {} & {}\ar[dl, "a"name=alp]\\
{} & v\ar[l, "b"name=bet] & {}\\
{} & {} & {}\ar[ul,  "c"name=alpp, swap]
\ar[-, shorten <= 3pt, shorten >= 3pt, from=alp, to=bet, bend right = 60, dashed]
\end{tikzcd} 

&

\begin{tikzcd}
{} & {} & {}\\
{} & v\ar[ul, "b"name=bet]\ar[dl, "d"name=betp, swap] & {}\ar[l, "a"name=alp]\\
{} & {} & {}
\ar[-, shorten <= 3pt, shorten >= 3pt, from=alp, to=bet, bend right = 60, dashed, swap]
\end{tikzcd}

& 

\begin{tikzcd}
{} & {} & {}\ar[dl, "a"name=alp]\\
{} & v\ar[ul, "b"name=bet]\ar[dl, "d"name=betp, swap] & {}\\
{} & {} & {}\ar[ul,  "c"name=alpp, swap]
\ar[-, shorten <= 3pt, shorten >= 3pt, from=alp, to=bet, bend right = 60, dashed]
\ar[-, shorten <= 3pt, shorten >= 3pt, from=alpp, to=betp, bend left = 60, dashed]
\end{tikzcd} 

\\

 Z\ni ba

&

bc\notin Z\ni ba

&

da\notin Z\ni ba

& 

bc,da\notin Z \ni ba,dc

\end{array}
\]
\end{defn}

Note that $a\neq c$ when $v$ is a tributary or a quadbutary, and $b\neq d$ when $v$ is a distributary or a quadbutary. 
Consider that these diagrams depict precisely the arrows in $Q$ which are incident at $v$. 
Of all the heads and tails of these arrows, note that we have only labeled $v$. 
For the remaining vertices, not only may they coincide, but it is possible that $a=b$ or $a=d$ (but not both).  

 \begin{example}
 \label{running-example-quiver}
      Let $Q$ be the quiver 
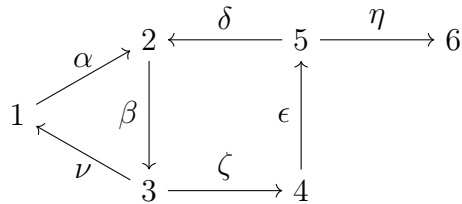
\begin{center}
\begin{tikzpicture}
	\node (1) at (-.732,1) {1};
	\node (2) at (1,2) {2};
	\node (3) at (1,0) {3};
	\node (4) at (3,0) {4};
	\node (5) at (3,2) {5};
	\node (6) at (5,2) {6};
	
	\draw[->] (1) to node[above]{$\alpha$} (2);
	\draw[->] (2) to node[left]{$\beta$} (3);
	\draw[->] (3) to node[below]{$\nu$} (1);
	\draw[->] (5) to node[above]{$\delta$} (2);
	\draw[->] (4) to node[left]{$\epsilon$} (5);
	\draw[->] (3) to node[above]{$\zeta$} (4);
	\draw[->] (5) to node[above]{$\eta$} (6);
\end{tikzpicture}
\end{center}
and let $Z=\{\beta\delta, \delta\epsilon, \epsilon\zeta\, \zeta\beta\}$. 
Here the vertices $1$ and $6$ are not relational, $2$ is a tributary, $3$ is a distributary, $4$ is a stream, and $5$ is another distributary. There are no quadbutaries. 
 \end{example}

\begin{defn}\label{defn-levee} \cite{Zembyk-Skewed-Gentle-A}
The \emph{levee} of $(Q,Z)$ at a relational vertex $v$ is a new pair $(Q^{v},Z^{v})$ that  `splits $Q$ at $v$' and is formally defined as follows. 
Let 
\[
\begin{array}{cc}
Q^{v}_{0}=\{u^{v}\mid u\in Q_{0}\setminus \{v\}\}\cup \{v(\sharp) \}\cup\{v(\flat) \},
&
Q^{v}_{1}=\{a^{v}\mid a\in Q_{1}\}.
\end{array}
\]
So vertices are found by relabelling the vertices in $Q$ other than $v$, and then adding a pair of new  and distinct vertices, denoted $v(\sharp) $ and $v(\flat) $, which replace $v$.  
Now let
\[
\begin{array}{ccc}
     h^{v}(a^{v})=h(a)^{v},
     & 
     t^{v}(a^{v})=t(a)^{v},
     & 
     (a\in Q_{1}, \,h(a)\neq v\neq t(a))
\end{array}
\]
 and otherwise, in the notation of \Cref{defn-relational}, given it makes sense, we let 
 \[
 \begin{array}{cc}
      t^{v}(b^{v})=v(\sharp) =h^{v}(c^{v}),
      & 
      t^{v}(d^{v})=v(\flat) =h^{v}(a^{v}).
 \end{array}
 \] 
    Hence the head and tail functions are essentially unchanged on arrows not incident at $v$, and otherwise, they are given according to the cases  below.
\[
\begin{array}{cccc}
\textbf{stream levee}

& 

\textbf{tributary levee}

&

\textbf{distributary levee}

&

\textbf{quadbutary levee}

\\

\begin{tikzcd}[column sep = 1.9em]
{} & v(\sharp)\ar[l, swap, "b^{v}"name=bet] &\\
& v(\flat)  & {}\ar[l, swap,  "a^{v}"name=alp]
\end{tikzcd}

&

\begin{tikzcd}[column sep = 1.9em]
{} & v(\sharp) \ar[l, swap, "b^{v}"name=bet] & {}\ar[l, swap,  "c^{v}"name=alpp]\\
{} & v(\flat)  & {}\ar[l, swap,  "a^{v}"name=alp]
\end{tikzcd}

&

\begin{tikzcd}[column sep = 1.9em]
{} & v(\sharp) \ar[l, swap, "b^{v}"name=bet] & {}\\
{} & v(\flat) \ar[l, swap, "d^{v}"name=betp] & {}\ar[l, swap,  "a^{v}"name=alp]
\end{tikzcd} 

& 

\begin{tikzcd}[column sep = 1.9em]
{} & v(\sharp) \ar[l, swap, "b^{v}"name=bet] & {}\ar[l, swap,  "c^{v}"name=alpp]\\
{} & v(\flat) \ar[l, swap, "d^{v}"name=betp] & {}\ar[l, swap,  "a^{v}"name=alp]
\end{tikzcd} 
\end{array}
\]
Note there is a choice of which of the two new vertices replacing $v$ is labeled $v(\sharp)$, and which is labeled $v(\flat)$, and so the new quiver $Q^{v}$ is defined only up to quiver isomorphism. 

  With the quiver $Q^{v}$ defined we now let 
  $Z^{v}=\{m^{v}n^{v}\mid mn\in Z,\,v\neq t(m),\, v\neq h(n)\}$. 
Thus $Z^v$ consists of the paths in $Z$ apart from those of the form $ba,dc$ from \Cref{defn-relational}. 
\end{defn}

\begin{example}
\label{example-ambiguous-notation-levee}
    Consider the gentle pair $(Q,Z_1)$ and the locally gentle pair  $(Q,Z_2)$ from  \Cref{example:gentlevslocallygentle}. 
    Let $v=2$. 
    Note that $v$ is relational with respect to $Z_{1}$ and $Z_{2}$, but by means of different relations. 
    Hence there is an abuse of notation, explained by the following examples. 
    The levee $(Q^{v},Z^{v}_{1})$ of $(Q,Z_1)$ at $v$ is given by $Z_{1}^{v}=\emptyset$ and 
    \begin{center}
        \begin{tikzpicture}
            \node (1) at (0,0) {$1$};
            \node (2) at (2,0) {$2(\sharp)$};
            \node (2') at (4,0) {$2(\flat)$};
            \node (3) at (6,0) {$3$};
            \draw[->] (1) -- node[above]{$\alpha$} (2);
            \draw[->] (2) --node[above]{$\beta$}(2');
            \draw[->] (2') -- node[above]{$\nu$} (3);
        \end{tikzpicture}
    \end{center}
while the levee $(Q^{v},Z^{v}_{2})$ is defined by $Z_{2}^{v}=\emptyset$ and the quiver 
    \begin{center}
        \begin{tikzpicture}
            \node (1) at (0,0) {$1$};
            \node (2) at (2,0) {$2(\sharp)$};
            \node (3) at (4,0) {$3$};
            \node (2') at (6,0) {$2(\flat)$};
            \draw[->] (1) -- node[above]{$\alpha$} (2);
            \draw[->] (2) --node[above]{$\nu$}(3);
            \draw[->] (2') .. controls (7.5,0.5) and (7.5, -0.5) .. node[pos=.5, right]{$\beta$} (2');
        \end{tikzpicture}
    \end{center}
Hence the symbol $Q^{v}$ conceals the way in which the vertex $v$ is relational. 
\end{example}

We first give with some immediate consequences of this definition. 

\begin{lem}\label{lem:ZembykBasics}
Let $v$ be a relational vertex in $Q$. 
The following statements hold. 
\begin{enumerate}[(i)]
    \item The pair $(Q^{v},Z^{v})$ is locally gentle.
    \item For the pair $(Q^{v},Z^{v})$ the new vertices $v(\sharp) ,v(\flat)\in Q_{0}^{v} $ are not relational.
\end{enumerate}
\end{lem}
\begin{proof}
Since $v$ is relational, there exist arrows $a,b$ such that $t(b)=v=h(a)$ and $ba\in Z$. 
As above let $v(\sharp) =t^{v}(b^{v})$ and $v(\flat) =h(a^{v})$ in the levee $Q^{v}$ of $Q$ at $v$. 

(i) By definition $Z^{v}$ consists of paths of length $2$. 
If $w=v(\sharp)$ or $w = v(\flat) $ then by construction there is at most $1$ arrow in $Q^{v}$ with head $w$ and at most $1$ arrow in $Q^{v}$ with tail $w$. 
Otherwise $w=u^{v}$ for $u$ a vertex in $Q$ with $u\neq v$, meaning there are at most $2$ arrows $c$ with $h(c)=u$ in $Q$, and so at most $2$ arrows $c^{v}$ in $Q^{v}$ with $h^{v}(c^{v})=u^{v}$. 
Hence, and dually, any vertex in $Q^{v}$ is the head (respectively, tail) of at most $2$ arrows in $Q^{v}$. 

Let $m^{v},n^{v},p^{v}$ be arrows in $Q^{v}$ corresponding to arrows $m, n , p $ in $Q$. 
Suppose that $t^{v}(m^{v})=h^{v}(n^{v})=h(p^{v})$, meaning that $t(m)=h( n )=h( p )$. 
If $t(m)\neq v$ and  $m^{v}p^{v},m^{v}n^{v}\in Z^{v}$ then $m p ,m n \in Z$,  implying that $ p = n $. 
Otherwise, when $t^{v}(m^{v})= v(\sharp)$ or $t^{v}(m^{(v)}) = v(\flat) $ it follows immediately that  $ p = n $, as depicted in \Cref{defn-levee}. 

This argument shows that for any arrow $m^{v}$ there is at most $1$ inadmissible path of the form $m^{v}n^{v}$. 
A similar argument shows that for any arrow $m^{v}$ there is at most $1$ admissible path of the form $m^{v}p^{v}$. 
Thus we have that for any arrow $m^{v}$ there is at most $1$ admissible (respectively, inadmissible) path of length $2$ ending with $m^{v}$. 
Dually, there exists at most $1$ admissible (respectively, inadmissible) path of length $2$ starting with $m^{v}$. 

(ii) 
Consider a path  $m^{v}n^{v}$ of length $2$ in $Q^{v}$ and let $w$ denote the vertex $t^{v}(m^{v})=h^{v}(n^{v})$. 
It follows by construction that if $w \in \{v(\flat),v(\sharp)\}$ then either   $m n \notin Z$ or  $t^{v}(m^{v})\neq h^{v}(n^{v})$. 
On the other hand, if $m^{v}n^{v}\in Z^{v}$ then $w\neq v(\sharp) $ and $w\neq v(\flat) $. 
\end{proof}

\Cref{lem:ZembykOrderDoesntMatter} shows that the order in which one takes successive levees does not matter.

\begin{lem}\label{lem:ZembykOrderDoesntMatter}
   Let $v,w\in Q_{0}$ be distinct and  relational. 
   Let $x=v^{w}\in Q^{w}_{0}$ and  $y=w^{v}\in Q^{v}_{0}$.  
   There is a quiver isomorphism $(Q^{w})^{x}\to (Q^{v})^{y}$ whose image of $(Z^{w})^{x}$ is $(Z^{v})^{y}$.  
\end{lem}

\begin{proof}
There are several cases. 
Without loss of generality we consider the most complicated case. 
So suppose $ b $ is an arrow with $h( b )=w$ and $t( b )=v$, and suppose additionally that $w$ and $v$ are both quadbutaries. 
The proof for the other cases where $v$ and $w$ are adjacent can be seen by removing the appropriate arrows from the diagrams in the sequel, and if $v$ and $w$ are not adjacent the result is immediate. 
In the case we are considering, the following diagram depicts the arrows    incident at $v$ together with those incident at  $w$. 
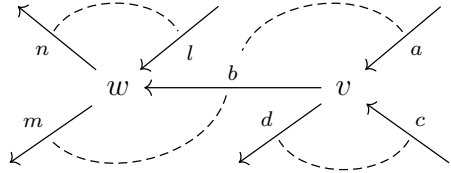
\[
\begin{tikzcd}
{} & {} & {}\ar[dl, "l"name=betpp] & {} &  {}\ar[dl, "a"name=alp]\\
{} & w\ar[ul,  " n"name=gam]\ar[dl, swap, " m "name=gamp] & {} &  v\ar[ll, swap, " b "name=bet]\ar[dl, swap, " d "name=betp] & {}\\
{} & {} & {} & {}  & {}\ar[ul, swap,  " c "name=alpp]
\ar[-, shorten <= 3pt, shorten >= 3pt, from=alp, to=bet, bend right = 60, dashed]
\ar[-, shorten <= 3pt, shorten >= 3pt, from=alpp, to=betp, bend left = 60, dashed]
\ar[-, shorten <= 3pt, shorten >= 3pt, from=bet, to=gamp, bend left = 60, dashed]
\ar[-, shorten <= 3pt, shorten >= 3pt, from=gam, to=betpp, bend left = 60, dashed]
\end{tikzcd} 
\]

The dashed arrows in the diagram indicate the following paths of length $2$: $ n  l$, $ m  b $, $ b a$ and $ d  c $. 
These paths constitute an exhaustive list of the relations $fg \in Z$ with $f,g\in Q_{1}$ and where $t(f)=h( g )$ is either $w$ or $v$. 

Let $x=v^{w}\in Q_{0}^{w}$ and let $y=w^{v}\in Q_{0}^{v}$. 
For the vertex $q=v\in Q_{0}$ (respectively, $q=x $) 
 in the quiver $P=Q$ (respectively, $P=Q^{w}$), 
as usual we use $q(\sharp)$ and $q(\flat)$ to denote the  vertices in $P^{q}$ which replace $q$. 
To avoid confusion by ambiguity, for the vertex $r=w$ (respectively, $r=y$)
 in the quiver $S=Q$ (respectively, $S=Q^{v}$),
we instead use $r(\Delta)$ and $r(\nabla)$ to denote the new vertices in $S^{r}$ which replace $r$. 
We stress that the symbols $\Delta$ and $\nabla$ are only used in this proof as placeholders for (either of) $\sharp$ and $\flat$. 

We now construct $(Q^{w},Z^{w})$ and  $(Q^{v},Z^{v})$. 
They are respectively, depicted on the left and right below. 
Note that these depictions are found by appropriately adjusting the previous diagram that depicted the arrows and relations in  $Q$ that involve $ b $.

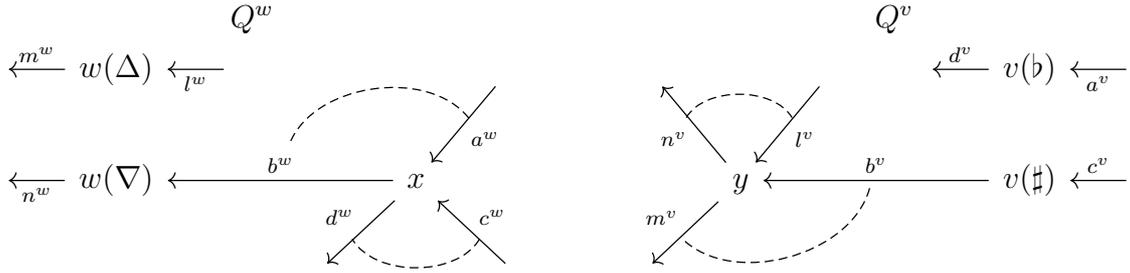
\[
\begin{array}{ccc}
Q^{w} 
&
&
Q^{v}\\
\begin{tikzcd}
[column sep = 0.75cm]
{} & w(\Delta)\ar[l,  " m^{w}"'name=gam] & {}\ar[l, "  l^{w}"name=betpp] & {} & {} &  {}\ar[dl, "a^{w}"name=alp]\\
{} & w(\nabla)\ar[l,  " n^{w}"name=gamp] & {} & {} &  x\ar[lll, swap, " b^{w}"name=bet]\ar[dl, swap, " d^{w}"name=betp] & {}\\
{} & {} & {} & {} & {}  & {}\ar[ul, swap,  " c^{w}"name=alpp]
\ar[-, shorten <= 3pt, shorten >= 3pt, from=alp, to=bet, bend right = 60, dashed]
\ar[-, shorten <= 3pt, shorten >= 3pt, from=alpp, to=betp, bend left = 60, dashed]
\end{tikzcd} 
& 
\hspace{.75cm}
&
\begin{tikzcd}[column sep = 0.75cm]
{} & {} & {}\ar[dl, "  l^{v}"name=betpp]  &{}& v(\flat) \ar[l, swap, "d^{v}"name=betp]  & {}\ar[l,   "a^{v}"name=alpp]
\\
{} & y\ar[ul,  " n^{v}"name=gam]\ar[dl, swap, " m^{v}"name=gamp] & {} &{}&  v(\sharp) \ar[lll, swap, "b^{v}"name=bet] & {}\ar[l, "c^{v}"'name=alp]\\
{} & {} & {} &{}& {}  & {}
\ar[-, shorten <= 3pt, shorten >= 3pt, from=bet, to=gamp, bend left = 60, dashed]
\ar[-, shorten <= 3pt, shorten >= 3pt, from=gam, to=betpp, bend left = 60, dashed]
\end{tikzcd}
\end{array}
\]
We now repeat this process to depict $(Q^{w})^{x}$ and $(Q^{v})^{y}$ as follows. 
Here for each of $u=v,w$, any $h\in Q_{1}$ and any $r\in Q^{u}_{0}$ we let $h^{u,r}$ denote the arrow $(h^{u})^{r}$ in $(Q^{u})^{r}$. 
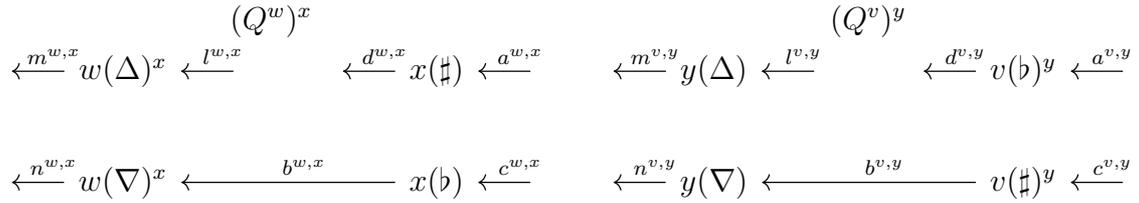
\[
\begin{array}{cc}
(Q^{w})^{x} 
&
(Q^{v})^{y}\\
\begin{tikzcd}
[column sep = 0.7cm]
{} & w(\Delta)^{x}\ar[l, swap,   "\quad m^{w,x}"name=gam] & {}\ar[l, swap, "\quad  l^{w,x}"name=betpp] & {} & x(\sharp)\ar[l, swap, "\quad d^{w,x}"name=betp] &  {}\ar[l, swap,  "\quad a^{w,x}"name=alp]\\
{} & w(\nabla)^{x}\ar[l, swap, "\quad n^{w,x}"name=gamp] & {} & {} &  x(\flat)\ar[lll, swap, "\quad b^{w,x}"name=bet] & {}\ar[l, swap,  "\quad c^{w,x}"name=alpp]
\end{tikzcd} 
& 
\begin{tikzcd}
[column sep = 0.7cm]
{} & y(\Delta)\ar[l, swap,   "\quad m^{v,y}"name=gam] & {}\ar[l, swap, "\quad  l^{v,y}"name=betpp] & {} & v(\flat)^{y}\ar[l, swap, "\quad d^{v,y}"name=betp] &  {}\ar[l, swap, "\quad a^{v,y}"name=alp]\\
{} & y(\nabla)\ar[l, swap, "\quad n^{v,y}"name=gamp] & {} & {} &  v(\sharp) ^{y}\ar[lll, swap, "\quad b^{v,y}"name=bet] & {}\ar[l, swap,  "\quad c^{v,y}"name=alpp]
\end{tikzcd} 
\end{array}
\]
Note that all the remaining vertices in $(Q^{w})^{x}$ have the form $u^{x}$ where $u\neq x,w$. 
Likewise all the remaining vertices in $(Q^{v})^{y}$ have the form $u^{y}$ where $u\neq y,v$. 
We define a mapping of vertices by $u^{x}\mapsto u^{y}$ for any $ u\neq x,w$, and otherwise
\[
\begin{array}{ccccc}
w(\Delta)^{x}\mapsto y(\Delta), 
&
w(\nabla)^{x}\mapsto y(\nabla), 
&
x(\sharp)\mapsto v(\flat)^{y}
&
x(\flat)\mapsto v(\sharp)^{y}.
\end{array}
\]
We define the assignment on arrows by $h^{w,x}\mapsto h^{v,y}$. 
By construction, this is an isomorphism sending $(Z^{w})^{x}$ to $(Z^{v})^{y}$ as required. 
\end{proof}

\begin{defn}
    \label{defn-zembyk-quiver-scissors} 
    By the \emph{Zembyk excision} of a locally gentle pair $(Q,Z)$ we mean a quiver which we denote $Q^{\zembyk}(Z)$ and define iteratively as follows.  Let $Q^{(0)}=Q$ and $Z^{(0)}=Z$. 

    Let $(Q^{(n)}, Z^{(n)})$ be a locally gentle pair for some given integer $n\geq0$.
    \begin{enumerate}
            \item[$(1;n)$]            \begin{enumerate}[(i)]
                \item If $Z^{(n)}= \emptyset $ then go to step $(2;n)(i)$. 
                \item If $Z^{(n)}\neq \emptyset $ then choose $v(n)\in Q^{(n)}_{0}$ relational and go to step $(2;n)(ii)$. 
            \end{enumerate}
            \item[$(2;n)$]\begin{enumerate}[(i)]
            \item Let $Q^{\zembyk}(Z)=Q^{(n)}$ and terminate the algorithm. 
                \item  Let $Q^{(n+1)}=(Q^{(n)})^{v(n)}$ and $Z^{(n+1)}=(Z^{(n)})^{v(n)}$ and go to step $(1;n+1)$. 
            \end{enumerate}
    \end{enumerate}
By \Cref{lem:ZembykBasics} we have that  $\vert Q^{(n+1)}_{0}\vert =\vert Q^{(n)}_{0}\vert +1$ and $\vert Z^{(n+1)}\vert <\vert Z^{(n)}\vert$ when $Z^{(n)}\neq \emptyset$. 
By \Cref{lem:ZembykOrderDoesntMatter} the output of the algorithm is independent of the choice of $v(n)$ in $(1;n)(ii)$.   
It follows that the algorithm terminates and is unambiguous.  
\end{defn} 

\begin{rem}
\label{remark-about-zembyks-excision}
    One motivation for \Cref{defn-zembyk-quiver-scissors} is as follows. 
    Let $(Q,Z)$ be a gentle pair. 
    Part of the  main result in work of Zembyk \cite[p. 649, Theorem]{Zembyk-Skewed-Gentle-A} says that the quiver $Q^{\zembyk}(Z)$ is a disjoint union of finitely many copies of quivers of Dynkin type $\BA$ or   $\tilde{\BA}$. 
    In fact this result  is  algebraic in its formulation. 
    We explain this in  \Cref{remark-about-zembyks-excision-algebraic}. 
\end{rem}

\begin{example}
    Consider the gentle pair $(Q,Z_1)$ and the locally gentle pair  $(Q,Z_2)$ from  \Cref{example:gentlevslocallygentle}. 
    In \Cref{example-ambiguous-notation-levee} we computed the levees of  $(Q,Z_1)$ and  $(Q,Z_2)$ with respect to the vertex $v=2$, which was the unique relational vertex in either case. 
    
    The quivers produced by the levees were therefore the Zembyk excisions $Q^{\zembyk}(Z_{1})$ and $Q^{\zembyk}(Z_{2})$ from \Cref{defn-zembyk-quiver-scissors}. 
    Thus the notation $Q^{\zembyk}(Z)$ is not ambiguous. 
\end{example}

\begin{example}\label{running-example:zembyk-excision}
   Recall the pair $(Q,Z)$ from \Cref{running-example-quiver}. 
In $(1;0)(b)$ we let $v(0)=2$. 
In $(1;1)(b)$ we let $v(1)=5^{v(0)}$. 
In $(1;2)(b)$ we let $v(2)=4^{v(1)}$. 
In $(1;3)(b)$ we let $v(3)=3^{v(2)}$. 
Hence $Q^{\zembyk}(Z)$ is the quiver from \Cref{fig-intro-quivers} in the introduction, namely
\begin{center}
\begin{tikzpicture}
	\node (1) at (-1.732,1) {1};
	\node (2) at (0,2) {$2(\sharp)$};
	\node (3) at (0,0) {$3(\sharp)$};
	\node (3') at (2,0) {$3(\flat)$};
	\node (2') at (2,2) {$2(\flat)$};
	\node (5) at (4,2) {$5(\sharp)$};
        \node (4') at (4,0) {$4(\flat)$};
         \node (4) at (6,0) {$4(\sharp)$};
	\node (5') at (6,2) {$5(\flat)$};
	\node (6) at (8,2) {6};
	
	\draw[->] (1) to node[above]{$\alpha$} (2);
	\draw[->] (2) to node[left]{$\beta$} (3);
	\draw[->] (3) to node[below]{$\nu$} (1);
	\draw[->] (5) to node[above]{$\delta$} (2');
	\draw[->] (4) to node[left]{$\epsilon$} (5');
	\draw[->] (3') to node[above]{$\zeta$} (4');
	\draw[->] (5') to node[above]{$\eta$} (6);
\end{tikzpicture}
\end{center}
\end{example}

\section{Surface dissections}\label{sec:SurfaceDissection}

Here we recall  topological notions we will need for our main results. 

\subsection{Surface dissections for locally gentle pairs}
\label{subsec-surfaces-and-locally-gentle-pairs}

We will later consider possibly disconnected surfaces, so our definition of a marked surface will allow this.

\begin{defn}\label{defn:arc} A \emph{marked surface} $(S,M)$ is either a pair consisting of a connected, oriented, Riemann surface  $S$ and a finite set of points $M\subset S$
such that each boundary component of $S$ contains at least one point in $M$, or a disjoint union of such surfaces, each with a set of marked points. 
An \emph{arc} $\gamma$ on $(S,M)$ is a curve considered up to homotopy which is contained in $S$, has both endpoints in $M$, and whose interior is disjoint from $\partial S\setminus M$.
\end{defn}

Let $e(\gamma, \delta)$ be the minimal number of crossings between any homotopic representatives of $\gamma$ and $\delta$, excluding endpoints. We say $\gamma$ and $\delta$ are \emph{compatible} if $e(\gamma, \delta) = 0$. 
Any collection $D$ of pairwise compatible arcs is called a \textit{dissection.} 
We define a \emph{face} of $D$ to be a connected component of $S \backslash D$. The set of all faces corresponding to the dissection $D$ is denoted $\mathcal{F}(D)$. A dissection $D$ is \emph{cellular} if every $F \in \mathcal{F}(D)$ is homeomorphic to a (unpunctured) disk. 

 Given a marked point $v \in M$ such that $v$ lies on the boundary of $S$, we define the \emph{neighbors} of $v$ to be the vertices which come immediately before or after the vertex as we follow the boundary component which $v$ lies on, using a fixed orientation. These neighbors can be equal to each other and they also can be equal to $v$. Marked points which instead lie on the interior of $S$ are called ``punctures''.

\begin{defn}\cite[Definition 3.5]{Palu-Pilaud-Plamondon-non-kissing-non-crossing}\label{def:DualDissections}
 Consider a marked surface $(S,M)$, and partition $M$ into $V \cup V^*$ such that the neighbors of every $v \in V$ which lie on $\partial S$ are in $V^*$ and similarly for any $f^* \in V^*$. Let $D$ be a cellular dissection for $(S,V)$ and similarly for $D^*$ with respect to $(S,V^*)$. Then, we say $D,D^*$ are \emph{dual} if 
 \begin{enumerate}
     \item there is a bijection between vertices in $V$ and faces in $\mathcal{F}(D^*)$ such that  for every $v \in V$, there is a unique face $v^* \in \mathcal{F}(D^*)$  such that $v$ lies in $v^*$ (either on the boundary or in the interior),
     \item there is a bijection between faces $f \in \mathcal{F}(D)$ and vertices $f^* \in V^*$ in the same manner as (1), and 
     \item for $v,w \in V$, there is an arc between $v$ and $w$ in $D$ which borders faces $f,g \in \mathcal{F}(D)$ if and only if there is an arc between vertices $f^*$ and $g^*$ in $D^*$ which borders faces $v^*,w^* \in \mathcal{F}(D^*)$.
 \end{enumerate}
\end{defn}

Given a dissection $D$ of $(S,V \cup V^*)$, its dual $D^*$ is uniquely determined up to homotopy. 

It follows from the definition of a pair of dual dissections that, for every $\tau \in D$, there exists a unique $\tau^* \in D^*$ such that $e(\tau,\tau^*) = 1$ and vice versa. Moreover, given $\tau \in D$ and $\mu^* \in D^*$ such that $\mu^* \neq \tau^*$, $e(\tau,\mu^*) = 0$.

\begin{example}
\label{ex:dissection}
In Figure \ref{fig:dissectionexample}, we give a disk $S$ with two sets $V = \{a,b,c,d,e,f\}$ and $V^* = \{i^*,ii^*,iii^*,iv^*,v^*,vi^*\}$ of marked points such that $V \cap \partial S$ and $V^* \cap \partial S$ alternate on $\partial S$ and two surface dissections $D$ and its dual $D^*$ which use marked points from $V$ and $V^*$ respectively. The vertices in $V$ and arcs in $D$ are in blue in (a) and gray in (b) and the vertices in $V^*$ and arcs in $D^*$ are in gray in (a) and in red in (b).
 The dissection $D$ is a cellular dissection of $(S,V)$ and similarly for $D^*$ and $(S,V^*)$. One can check that $D$ and $D^*$ are dual dissections, and the vertices, arcs, and faces of each have been labeled to highlight the symmetry of Definition \ref{def:DualDissections}.

    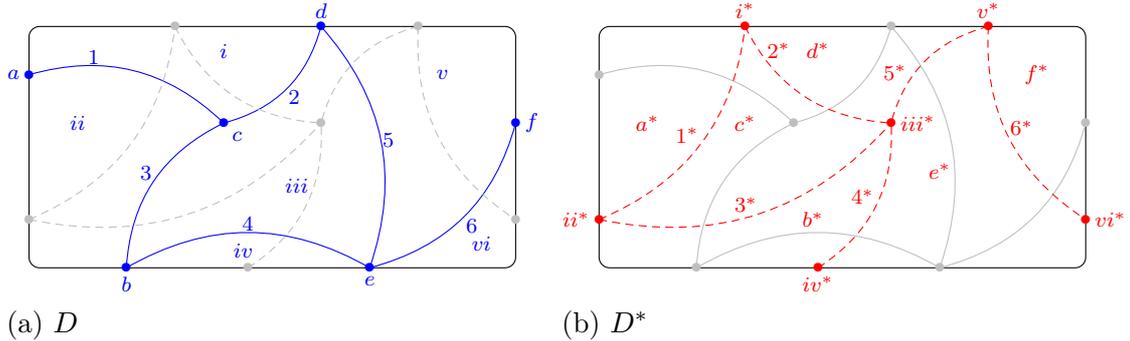
\begin{figure}[H]
        \centering
        \begin{subfigure}[b]{0.45\textwidth}            
            \begin{tikzpicture}[outer sep =0.3, inner sep =.4, scale= 0.64, font=\tiny]
            	
\draw[rounded corners] (0,1)--(0,5)--(10,5)--(10,0)--(0,0)--(0,1);

                \begin{scope}[blue]

\node[label=left:$a$]  (a) at (0,4) {$\bullet$};
\node[label=below:$b$] (b) at (2,0) {$\bullet$};
\node[label=below right:$c$] (c) at (4,3) {$\bullet$};
\node[label=above:$d$] (d) at (6,5) {$\bullet$};
\node[label=below:$e$] (e) at (7,0) {$\bullet$};
\node[label=right:$f$] (f) at (10,3) {$\bullet$};

\draw (a.center) to[bend left] node[above, pos=0.3] {1} (c.center);
\draw (d.center) to[bend left] node[below right] {2} (c.center);
\draw (b.center) to[bend left] node[above left] {3} (c.center);
\draw (b.center) to[bend left] node[above] {4} (e.center);
\draw (d.center) to[bend left] node[right] {5} (e.center);
\draw (f.center) to[bend left] node[below right] {6} (e.center);

\node at (4,4.5) {$i$};
\node at (1,3) {$ii$};
\node at (5.5,1.7) {$iii$};
\node at (4.4,.35) {$iv$};
\node at (8.5,4) {$v$};
\node at (9.3,.5) {$vi$};
\end{scope}

\begin{scope}[lightgray, densely dashed]
\node (i) at (3,5) {$\bullet$};
\node (ii) at (0,1) {$\bullet$};
\node (iii) at (6,3) {$\bullet$};
\node (iv) at (4.5,0) {$\bullet$};
\node (v) at (8,5) {$\bullet$};
\node (vi) at (10,1) {$\bullet$};

\draw (ii.center) to[bend right] (i.center);
\draw (i.center) to[bend right] (iii.center);
\draw (ii.center) to[bend right] (iii.center);
\draw (iv.center) to[bend right] (iii.center);
\draw (v.center) to[bend right] (iii.center);
\draw (v.center) to[bend right] (vi.center);
\end{scope}
            \end{tikzpicture}
            \caption{$D$}
        \end{subfigure}
\begin{subfigure}[b]{0.45\textwidth}            
    \begin{tikzpicture}[outer sep =0.3, inner sep =.4, scale= 0.64, font=\tiny]
        \draw[rounded corners] (0,1)--(0,5)--(10,5)--(10,0)--(0,0)--(0,1);
                \begin{scope}[lightgray]
\node (a) at (0,4) {$\bullet$};
\node (b) at (2,0) {$\bullet$};
\node (c) at (4,3) {$\bullet$};
\node (d) at (6,5) {$\bullet$};
\node (e) at (7,0) {$\bullet$};
\node (f) at (10,3) {$\bullet$};

\draw (a.center) to[bend left] (c.center);
\draw (d.center) to[bend left] (c.center);
\draw (b.center) to[bend left] (c.center);
\draw (b.center) to[bend left] (e.center);
\draw (d.center) to[bend left] (e.center);
\draw (f.center) to[bend left] (e.center);
\end{scope}
\begin{scope}[red, densely dashed]
\node[label=above:$i^*$] (i) at (3,5) {$\bullet$};
\node[label=left:$ii^*$] (ii) at (0,1) {$\bullet$};
\node[label=right:$iii^*$] (iii) at (6,3) {$\bullet$};
\node[label=below:$iv^*$] (iv) at (4.5,0) {$\bullet$};
\node[label=above:$v^*$] (v) at (8,5) {$\bullet$};
\node[label=right:$vi^*$] (vi) at (10,1) {$\bullet$};

\draw (ii.center) to[bend right] node[above left] {$1^*$} (i.center);
\draw (i.center) to[bend right] node[above right, pos=0.2] {$2^*$}(iii.center);
\draw (ii.center) to[bend right] node[above left] {$3^*$}(iii.center);
\draw (iv.center) to[bend right] node[above left] {$4^*$}(iii.center);
\draw (v.center) to[bend right] node[above left, pos=0.7] {$5^*$}(iii.center);
\draw (v.center) to[bend right] node[above right] {$6^*$}(vi.center);

\node at (1,3) {$a^*$};
\node at (4.4,1) {$b^*$};
\node at (3,3) {$c^*$};
\node at (4.5,4.5) {$d^*$};
\node at (7,2) {$e^*$};
\node at (9,4) {$f^*$};
\end{scope}
            \end{tikzpicture}
            \caption{$D^*$}
        \end{subfigure}
        \caption{A dissection $D$ and its dual $D^*$, with associated labels as in \cref{ex:dissection}. Note the correspondence between marked points and dual faces; and vice versa. }
        \label{fig:dissectionexample}
    \end{figure}

\end{example}

  Given a locally gentle pair $\overline{Q} = (Q,Z)$, Palu, Pilaud, and Plamondon associate a surface $S_{\overline{Q}}$ with marked points partitioned into disjoint sets $V \cup V^*$ and a pair of dual cellular dissections $D_{\overline{Q}}$ and $D^*_{\overline{Q}}$, where arcs in $D_{\overline{Q}}$ $(D^*_{\overline{Q}})$ have endpoints in $V$ $(V^*$) respectively. 
We refer the reader to \cite{Palu-Pilaud-Plamondon-non-kissing-non-crossing} for a full description of the model.
This construction is dual to the following method to recover a locally gentle pair $\overline{Q}_D = (Q_D,Z_D)$  from a surface $S$ and cellular dissection $D$.
\begin{enumerate}
    \item There is one vertex $i \in Q_0$ for each $\tau_i \in D$.
    \item There is one arrow $a \in Q_1$ with $t(a) = i$ and $h(a) = j$ whenever we have a pair of arcs $\tau_i,\tau_j$ which share a common endpoint and at this endpoint, $\tau_j$ immediately follows $\tau_i$ in clockwise order. 
    \item There is a product of two distinct arrows $ba \in Z$ for every pair of arrows such that  $t(b) = h(a)$ and $\tau_{h(b)}, \tau_{t(b)}, \tau_{t(a)}$ border the same face of $D$.
\end{enumerate}

Given a vertex $i \in Q_0$, we denote the associated arc in $D_{\overline{Q}}$ as $\tau_i$ and the unique arc from $D^*_{\overline{Q}}$ which crosses $\tau_i$ as $\tau_i^*$.  

Palu, Pilaud, and Plamondon show in Proposition 4.5 in \cite{Palu-Pilaud-Plamondon-non-kissing-non-crossing} that, if we swap $D$ for its dual $D^*$, this amounts to constructing the surface with dissection for $(Q,Z')$ where $Z'$ is the complement of $Z$ in the set of all paths of length $2$ in $Q$.

\begin{example}\label{Ex:DissectionToQuiver}
The dissection $D$ in Figure \ref{fig:dissectionexample} corresponds to the locally gentle pair $(Q,Z)$ from \Cref{running-example:zembyk-excision}. One can, for example, see that we have one 4-cycle with all subpaths of length $2$ in $Z_D$ associated to the \emph{internal} face $iii$ of $D$, where by internal we mean that all boundary edges of this face come from $D$. The 3-cycle in $(Q,Z)$ with no relations is associated to the puncture $c$ (which gives an internal face $c^*$ of $D^*$). 
\end{example}

\begin{rem}
A pair $(Q,Z)$ is a gentle pair, so that it has finitely many admissible paths, exactly when all endpoints of arcs in the associated dissection $D$ lie on the boundary of $S$. Surface models for gentle pairs were given in \cite{Baur-Coelho-Simoes-geometric-model-module-cat-gentle} and \cite{opper2018geometric}. 
\end{rem}

 Given a locally gentle pair $\overline{Q} = (Q,Z)$, with an associated surface and pair of dual dissections $D_{\overline{Q}}, D_{\overline{Q}}^*$, let $R^*_{\overline{Q}} = \{\tau_i^* \in D_{\overline{Q}}^* \colon  \exists ba \in Z \text{ such that } t(b) = h( a ) = i\}$.  In other words, $R^*_{\overline{Q}}$ is the subset of $D^*_{\overline{Q}}$ consisting of all arcs which correspond to a relational vertex in $\overline{Q}$. As an example, Figure \ref{fig:relation_dissection} shows the set $D_{\overline{Q}} \cup R^*_{\overline{Q}}$ for the pair of dual dissections $D_{\overline{Q}}, D^*_{\overline{Q}}$ associated to the quiver in Example \ref{running-example:zembyk-excision}.

\begin{figure}[H]
    \centering
    \begin{tikzpicture}[inner sep=1, outer sep = 0]
        \draw[rounded corners] (0,1)--(0,5)--(10,5)--(10,0)--(0,0)--(0,1);
        
        \begin{scope}[blue]
        \node[label=left:$a$]  (a) at (0,4) {$\bullet$};
        \node[label=below:$b$] (b) at (2,0) {$\bullet$};
        \node[label=below right:$c$] (c) at (4,3) {$\bullet$};
        \node[label=above:$d$] (d) at (6,5) {$\bullet$};
        \node[label=below:$e$] (e) at (7,0) {$\bullet$};
        \node[label=right:$f$] (f) at (10,3) {$\bullet$};

        \draw (a.center) to[bend left] node[above right] {1} (c.center);
        \draw (d.center) to[bend left] node[below right] {2} (c.center);
        \draw (b.center) to[bend left] node[above left] {3} (c.center);
        \draw (b.center) to[bend left] node[above right] {4} (e.center);
        \draw (d.center) to[bend left] node[below right] {5} (e.center);
        \draw (f.center) to[bend left] node[below right] {6} (e.center);
        
        
        \end{scope}
        \begin{scope}[red, densely dashed]
        \node (i) at (3,5) {$\bullet$};
        \node (ii) at (0,1) {$\bullet$};
        \node (iii) at (6,3) {$\bullet$};
        \node (iv) at (4.5,0) {$\bullet$};
        \node (v) at (8,5) {$\bullet$};
        \node (vi) at (10,1) {$\bullet$};
        
        \draw (i.center) to[bend right] (iii.center);
        \draw (ii.center) to[bend right] (iii.center);
        \draw (iv.center) to[bend right] (iii.center);
        \draw (v.center) to[bend right] (iii.center);
        \end{scope}
         
        \begin{scope}[lightgray, densely dashed]
        \draw (ii.center) to[bend right] (i.center);
        \draw (v.center) to[bend right] (vi.center);
        \end{scope}
        
    \end{tikzpicture}

    \caption{The arcs from the dissection $D_{\overline{Q}}$ for $\overline{Q}$ from Example \ref{running-example:zembyk-excision}  are solid and blue, the arcs from $R^*_{\overline{Q}_D}=\{\tau^*_2,\tau^*_3,\tau^*_4,\tau^*_5\}$ are dashed and red, and the arcs $D^*_{\overline{Q}} \backslash R^*_{\overline{Q}}$ are dashed and gray. }
    \label{fig:relation_dissection}
\end{figure}
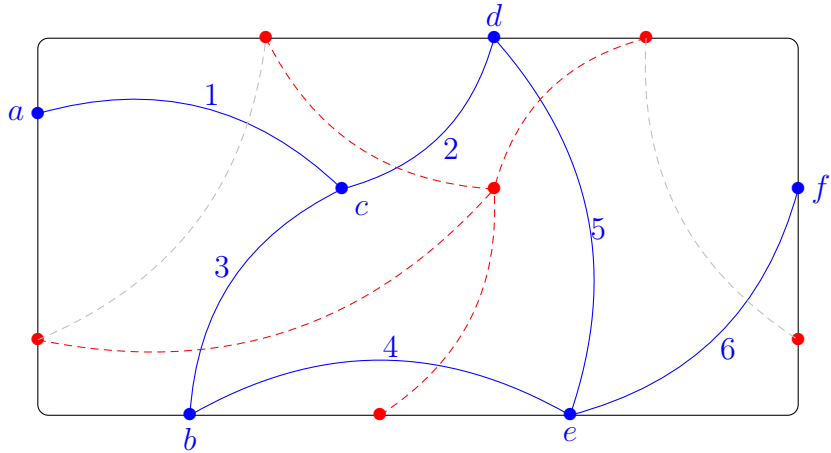

 In the following, since we consider both a dissection $D$ and a subset of its dual dissection $R^*$ we will have some pairs of arcs which cross. We will still refer to a \emph{face} as a connected component of $S\backslash(D \cup R^*)$, i.e. an element of $\mathcal{F}(D \cup R^*)$. Note that some of the bounding arcs on a face in $\mathcal{F}(D \cup R^*)$ could intersect in a point which is not in $V \cup V^*$.

\begin{lem}\label{lem:DivideIntoSmallerFaces}
Given a surface with dual dissections $(S_{\overline{Q}},V_{\overline{Q}} \cup V^*_{\overline{Q}}, D_{\overline{Q}}, D^*_{\overline{Q}})$ associated to a locally gentle pair $\overline{Q}$, every face in $\mathcal{F}(D_{\overline{Q}} \cup R^*_{\overline{Q}})$ is bounded by at most two arcs from $D_{\overline{Q}}$ and, if there are exactly two, they share an endpoint which is also lies on the face.
\end{lem}

\begin{proof}
Applying the same logic as in Proposition 1.12 in \cite{opper2018geometric} to the locally gentle case, we see that there are two types of faces formed from $D_{\overline{Q}}$: internal faces and those which have one edge along the boundary of $S_{\overline{Q}}$. 

Consider a face which has one edge on the boundary and  $m \geq 1$ edges from $D_{\overline{Q}}$. The corresponding vertex from $V^*$ lies on this boundary edge. If $m = 1$ or $2$, then it is possible that the bounding edge(s) from $D_{\overline{Q}}$ correspond to non-relational vertices from $\overline{Q}$. But if $m \geq 3$, then the edges from $D_{\overline{Q}}$ which don't share a vertex with the boundary component on this face correspond to relational vertices in $\overline{Q}$. If we add the corresponding dual arcs, we will divide this face so that, now, each sub-face is bounded by two arcs from $D_{\overline{Q}}$ and these are consecutive.

Next, consider an internal face. This corresponds to a directed cycle in $Q$ with all subpaths of length $2$ in the ideal $I$; therefore, this is a cycle with no relations in $(Q,Z')$ where $Z'$ is the complement of $Z$ in the set of all length two paths in $Q$. This means that there is a puncture from $V^*$ in the interior of this face and arcs incident to this puncture $\tau^*$ for each $\tau \in D_{\overline{Q}}$ bordering this internal face. For an illustration, see Figure \ref{fig:relation_dissection}, and the face bounded by the arcs $2, 3, 4$ and $5$. Since we add all such $\tau^*$ to $R_{\overline{Q}}^*$, the desired property is preserved. 

\end{proof}

A consequence of our method of breaking up the faces of $D_{\overline{Q}}$ is that every face corresponds to at most one arrow from $\overline{Q}$.

\begin{cor}\label{cor:ArrowsToFaces}
Given a locally gentle pair $\overline{Q}$, there is an injection, $\mathcal{L}$, from the set of arrows in $Q$ to $\mathcal{F}(D_{\overline{Q}} \cup R^*_{\overline{Q}})$.
\end{cor}

For example, returning to Figure \ref{fig:relation_dissection}, the face bounded by arcs $\tau_2,\tau_3,\tau_2^*,$ and $\tau_3^*$ would be $\mathcal{L}(\beta)$ for the arrow $\beta$ from $Q$ in Example \ref{running-example:zembyk-excision}. Note that in Figure \ref{fig:dissectionexample}, if we only consider $\mathcal{F}(D_{\overline{Q}})$, then some faces will correspond to multiple arrows from $Q$. For example, the face bounded by $\tau_2,\tau_3,\tau_4$, and $\tau_5$ corresponds to arrows $\beta,\zeta,\epsilon,$ and $\delta$. 

\subsection{Permissible arcs and closed curves}

We will consider arcs with respect to a cellular dissection $D$. The dual dissection $\overline{D}$ will not affect our definition of permissibility. 

We note that when $\overline{Q}$ is gentle, our surface with dissection $D$ resembles that from \cite{Baur-Coelho-Simoes-geometric-model-module-cat-gentle}, except we do not allow the unmarked boundary components. In their place, we retain the points from $V^*$ which lie on the interior of $S$. For simplicity, we keep these points in all internal faces, while in \cite{Baur-Coelho-Simoes-geometric-model-module-cat-gentle} the unmarked boundary components only occur inside internal faces with 1 or 2 bounding arcs. 

\begin{defn}\label{def:PermissibleArc}
Given an arc $\gamma$ on a surface $(S,V\cup V^*)$ with a pair of dual dissections $D$ and $D^*$, choose an orientation of $\gamma$ and let $\rho_1,\ldots,\rho_d$ be the sequence of arcs from $D$ which $\gamma$ crosses. Similarly, let $F_0,\ldots,F_d$ be the sequence of faces of $D$ which $\gamma$ passes through, where $\rho_1$ is on the boundary of $F_0$, $\rho_j$ and $\rho_{j+1}$ are on the boundary of $F_j$ for $1 \leq j \leq d-1$, and $\rho_d$ is on the boundary of $F_d$. We say $\gamma$ is \emph{permissible} if \begin{enumerate}
    \item the endpoints of $\gamma$ are in $V$,
    \item if $F_0$ or $F_d$ contains a puncture from $V^*$, $\gamma$ winds around this puncture at most once, and $\gamma$ does not wind around a puncture from $V^*$ contained in $F_i, 0 < i < d$,
    \item for each $1 \leq j \leq d-1$, the arcs $\rho_j$ and $\rho_{j+1}$ share a common endpoint $p_j$, and
    \item for each $1 \leq j \leq d-1$, the  disk cut out by $\gamma,\rho_j,$ and $\rho_{j+1}$ which contains $p_j$ on its boundary does not include any point of $V^*$.
\end{enumerate}
We say a closed curve $\xi$ is \emph{permissible} if it does not wind around any punctures from $V^*$, satisfies conditions (3) and (4), is not contractible, and does not cut out a $m$-punctured disk with all punctures from $V^*$ for any $m \geq 1$.
\end{defn}

We remark that this definition also resembles the definition of a $D$-accordion in \cite{Palu-Pilaud-Plamondon-non-kissing-non-crossing}.

Our definition of two \emph{equivalent arcs} will exactly match Definition 3.5 of \cite{Baur-Coelho-Simoes-geometric-model-module-cat-gentle}, with the adjustment that their tiles of Type I and Type II are for us internal faces bounded by 1 and 2 arcs respectively, and the unmarked boundary components in the aforementioned work are here replaced by punctures from $V^*$. Note we will later refer to faces of type 1 and 2, which will be distinct from this notion.

There are further considerations when considering which pairs of closed curves will be equivalent.


\begin{defn}\label{def:EquivClosedCurves} (c.f. Br\"{u}stle and Zhang \cite[Proof of Theorem 1.1, p.537]{Brustle-Zhang-cluster-category}). 
Let $S$ be a surface and let $\gamma\colon  [0,1] \to S$ be a closed curve on $S$. 

For an integer $n\geq 1$ we define the $n$-\emph{fold repetition} of $\gamma$ to be the closed curve $\gamma^{n}\colon [0,1]\to S $ defined by  $\gamma^{n}(t)=\gamma(nt-m)$ whenever $t\in[\frac{m}{n},\frac{m+1}{n} ]$ for an integer $m$ with $0\leq m\leq n-1$. 
By a \emph{repetition} of $\gamma$ we mean the $n$-fold repetition for some  $n\geq 1$. 
We say that $\gamma$ is \emph{primitive} if it is not homotopy equivalent to  the $n$-fold repetition of another closed curve for some integer $n\geq 2$; see  \cite[Definition 1.19]{opper2018geometric}.

Define an equivalence relation $\sim$ on the set of closed curves $[0,1]\to S$ by the symmetric closure of the  transitive closure  of the reflexive relation consisting of pairs $\alpha,\beta\colon  [0,1] \to S$ of closed curves such that a repetition of $\alpha$ is homotopy equivalent to a repetition of $\beta$. 
We say closed curves $\gamma,\delta\colon  [0,1] \to S$ are \emph{repetition equivalent} if $\gamma\sim\delta$. 
Hence the equivalence class of $\delta$ with $\gamma\sim \delta$  may be represented by a primitive closed curve. 
\end{defn}

\section{Geometric version of Zembyk's algorithm}\label{sec:geozembyk}
In \S\ref{sec:geozembyk} we provide a geometric description of Zembyk's algorithm from  \S\ref{subsec:zembykalgorithm}. Namely, we realise Zembyk's algorithm on locally gentle pairs as cutting along the arcs in $R^*$ in the surface model. 
We provide two examples: the first, in \Cref{fig:cutting_example_simple}, is more straightforward, to illustrate the procedure; the second, in \Cref{fig:cutting_example_running}, is more complicated, to display potential behaviour.  
For the first example we fix  a surface dissection, together with its dual, in \Cref{fig:cutting-original}.   
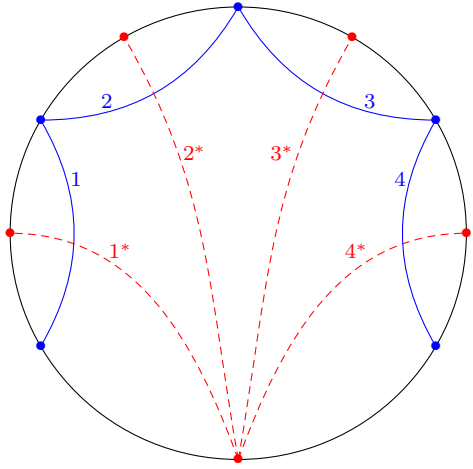
\begin{figure}[H]
	\centering
	\begin{tikzpicture}[font=\tiny, inner sep =.5]
		\draw(0,0) circle[radius=3];
		\begin{scope}[blue]
		\node (a) at (210:3) {$\bullet$};
		\node (b) at (150:3) {$\bullet$};
		\node (c) at (90:3) {$\bullet$};
		\node (d) at (30:3) {$\bullet$};
		\node (e) at (330:3) {$\bullet$};
		\draw (a.center) to [bend right] node[above right, pos=0.7] {$1$} (b.center);
		\draw (b.center) to [bend right] node[above left, pos=0.3] {$2$}(c.center);
		\draw (c.center) to [bend right] node[above right, pos=0.7] {$3$}(d.center);
		\draw (d.center) to [bend right] node[above left, pos=0.3] {$4$}(e.center);
		\end{scope}
		\begin{scope}[red, densely dashed]
		\node (base) at (270:3) {$\bullet$};
		\node (1) at (180:3) {$\bullet$};
		\node (2) at (120:3) {$\bullet$};
		\node (3) at (60:3) {$\bullet$};
		\node (4) at (0:3) {$\bullet$};
		\draw (base.center) to[out=110,in=0] node[above right, pos=0.7]{$1^*$} (1.center);
		\draw (base.center) to[out=100,in=300] node[above right, pos=0.7]{$2^*$} (2.center);
		\draw (base.center) to[out=80,in=240] node[above left, pos=0.7]{$3^*$} (3.center);
		\draw (base.center) to[out=70,in=180] node[above left, pos=0.7]{$4^*$} (4.center);
		\end{scope}
	\end{tikzpicture}
	\caption{The running example of a surface dissection together with its dual; in this case, a dissected polygon.}
  \label{fig:cutting-original}
  \end{figure}
 Consider a (not necessarily connected) surface $(S,M)$ with $M = V \cup V^*$, and a set of dissections $D, R^*$ such that $D$ is a cellular dissection of $(S,V)$ and $R^*$ is a dissection of $(S,V^*)$ such that $R^* \subseteq D^*$ where $D^*$ is the dual dissection to $D$. Let $B \subseteq R^*$. We define the \emph{split of $(S,M,D,R^*)$ with respect to $B$} to be the collection of surfaces, each endowed with a pair of dissections $\{(S_i, V_{D_i}\cup V_{D_i^*},D_i, R_i^*)\}_i$ where for each $i$,
\begin{itemize}
\item  $S_i \in \mathcal{F}(B)$, with boundary given by the union of the boundary of $S_i$ and the bounding arcs of $B$ in this face,
\item $V_{D_i}$ consists of all marked points in $V \cap S_i$ as well as any intersection points $D \cap B$ which lie in $S_i$,
\item $V_{R_i}^*$ consists of all points from $V^*$ which lie in $S_i$ and are incident to at least one arc in $R^* \backslash B$, 
\item $D_i$ is the set of arcs $D$ with non-empty intersection with $S_i$, possibly endowed with an endpoint from $D \cap B$ (formal definition below), and 
\item $R_i^* = (R^* \backslash B) \cap S_i$.
\end{itemize}
In Figure \ref{fig:cutting_example_simple}, we compute and draw the split of the set of dual dissections from \Cref{fig:cutting-original} where $R^* = D^*$ and $B = \{2^*\}$. This example also shows that, even if we start with a pair of dual dissections, the resulting dissections $D_i$ and $R_i^*$ will not be dual. 
In particular, the image of any arc $\tau \in D$ which is crossed by an arc in $B$ in the split will not be crossed by any other arcs.

\begin{figure}[H]
	\begin{subfigure}[b]{0.3\textwidth}   
	\centering
	\begin{tikzpicture}[font=\tiny, inner sep =.5, scale=.7]
		
		\draw (270:3) arc[start angle=270, end angle=120, radius=3];
		\begin{scope}[blue]
			\node (a) at (210:3) {$\bullet$};
			\node (b) at (150:3) {$\bullet$};
			\node (c) at (122:2) {$\bullet$};
		\draw (a.center) to [bend right]  (b.center);
		\draw (b.center) to [bend right] (c.center);
		\end{scope}
		\begin{scope}[red, densely dashed]
		\node (base) at (270:3) {$\bullet$};
		\node (1) at (180:3) {$\bullet$};
		\node (2) at (120:3) {$\bullet$};
		\draw (base.center) to[out=110,in=0] node[above right, pos=0.7]{$1^*$} (1.center);
		\draw (base.center) to[out=100,in=300] node[above right, pos=0.5]{$2^*$} (2.center);
		\end{scope}
	\end{tikzpicture}
 	\caption{Left half after splitting}
  	\end{subfigure}
	\begin{subfigure}[b]{0.3\textwidth}  
	\centering
	\begin{tikzpicture}[font=\tiny, inner sep =.5, scale=.7]
		\draw[lightgray,dotted](0,0) circle[radius=3];
		\draw (270:3) arc[start angle=90, end angle=300, radius=-3];
		\begin{scope}[blue]
		\node (b) at (122:2) {$\bullet$};
		\node (c) at (90:3) {$\bullet$};
		\node (d) at (30:3) {$\bullet$};
		\node (e) at (330:3) {$\bullet$};
		\draw (b.center) to [bend right] (c.center);
		\draw (c.center) to [bend right] (d.center);
		\draw (d.center) to [bend right] (e.center);
		\end{scope}
		\begin{scope}[red, densely dashed]
		\node (base) at (270:3) {$\bullet$};
		\node (2) at (120:3) {$\bullet$};
		\node (3) at (60:3) {$\bullet$};
		\node (4) at (0:3) {$\bullet$};
		\draw (base.center) to[out=100,in=300] node[left, pos=0.6]{$2^*$} (2.center);
		\draw (base.center) to[out=80,in=240] node[left, pos=0.6]{$3^*$} (3.center);
		\draw (base.center) to[out=70,in=180] node[above left, pos=0.7]{$4^*$} (4.center);
		\end{scope}
	\end{tikzpicture}
	\caption{Right half after splitting}
 	\end{subfigure}
  
	\begin{subfigure}[b]{0.3\textwidth} 
	\centering  
	\begin{tikzpicture}[font=\tiny, inner sep =.5,scale=.6]
		\draw(0,0) circle[radius=3];
		\begin{scope}[blue]
			\node (a) at (225:3) {$\bullet$};
			\node (b) at (135:3) {$\bullet$};
			\node (c) at (45:3) {$\bullet$};
		\draw (a.center) to [bend right]  (b.center);
		\draw (b.center) to [bend right] (c.center);
		\end{scope}
		\begin{scope}[red,densely dashed]
		\node (base) at (270:3) {$\bullet$};
		\node (1) at (90:3) {$\bullet$};
		\draw (base.center) to[bend right] node[above right, pos=0.5]{$1^*$} (1.center);
		\end{scope}
	\end{tikzpicture}
	\label{subfig:cutting-new-1}
	\caption{New surface, corresponding to left  hand side}
 	\end{subfigure}
	\begin{subfigure}[b]{0.3\textwidth}   
	\centering
	\begin{tikzpicture}[font=\tiny, inner sep =.5, scale=.6]
		\draw (0,0) circle[radius=3];
		\begin{scope}[blue]
		\node (b) at (210:3) {$\bullet$};
		\node (c) at (135:3) {$\bullet$};
		\node (d) at (45:3) {$\bullet$};
		\node (e) at (330:3) {$\bullet$};
		\draw (b.center) to [bend right] (c.center);
		\draw (c.center) to [bend right] (d.center);
		\draw (d.center) to [bend right] (e.center);
		\end{scope}
		\begin{scope}[red, densely dashed]
		\node (base) at (270:3) {$\bullet$};
		\node (3) at (90:3) {$\bullet$};
		\node (4) at (0:3) {$\bullet$};
		\draw (base.center) to[bend left] node[left, pos=0.6]{$3^*$} (3.center);
		\draw (base.center) to[bend left] node[above left, pos=0.7]{$4^*$} (4.center);
		\end{scope}
	\end{tikzpicture}
		\label{subfig:cutting-new-2}
	\caption{New surface, corresponding to right  hand side}
 	\end{subfigure}
 	\caption{Split of the surface in \Cref{fig:cutting-original} with respect to one arc, $2^*$.}
  \label{fig:cutting_example_simple}
\end{figure}
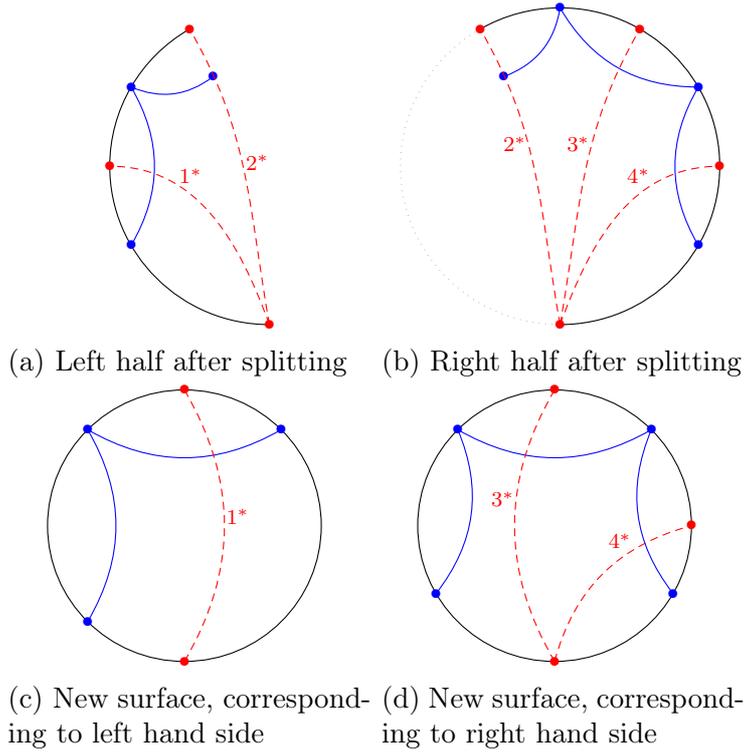

Next we provide a formal description of the split of the dissection $D$. 

\emph{Formal description of $D_i$. }
All arcs from $D$ which lie entirely in $S_i$ are in $D_i$. Furthermore, 
for each $\tau_v^* \in B \subseteq D^*$, the arc $\tau_v$ will be cut up in the splitting. Considering $\tau_v$ formally $\tau_v\colon  [0,1] \to S$, we choose a parametrization such that the intersection point $\tau_v \cap \tau_v^*$ occurs at $\tau_v(\frac12)$. Suppose that $\tau_v^*$ becomes a boundary edge for connected components $S_i$ and $S_j$ where $\tau_v(0) \in S_i$ and $\tau_v(1) \in S_j$ (note that it is possible that $S_i = S_j$). Then, if we define $\tau_v^{i}\colon  [0,1] \to S_i$ by $\tau_v^i(t) = \tau_v(\frac{t}{2})$ and $\tau_v^{j}\colon [0,1] \to S_j$ by $\tau_v^j(\frac{t}{2} + \frac12)$ we include $\tau_v^i \in D_i$ and $\tau_v^j \in D_j$. Note that the points $\tau_v^i(\frac12)$ and $\tau_v^j(\frac12)$ are included in $V_{D_i}$ and $V_{D_j}$ respectively. See for example the intersection between arcs $\tau_2$ and $\tau_{2^*}$ in Figure \ref{fig:cutting_example_simple}.

\Cref{fig:cutting_example_running} shows the result of splitting the marked surface with dissection $(S_{\overline{Q}}, V_{\overline{Q}} \cup V^*_{\overline{Q}}, D_{\overline{Q}}, R^*_{\overline{Q}})$ from Figure \ref{fig:relation_dissection} with $R^* = B$.

\begin{figure}[H]
	\begin{subfigure}[c]{0.24\textwidth}   
	\begin{tikzpicture}[font=\tiny, inner sep =.5]
		\draw[rounded corners] (0,1) -- (0,0) -- (3,0) -- (3,2) --(0,2) --(0,1);
		\begin{scope}[blue]
			\node[label=left:$a$] (a) at (0,1) {$\bullet$};
			\node (b) at (1,0) {$\bullet$};
			\node[label=below:$c$] (c) at (1.7,1.2) {$\bullet$};
			\node (d) at (3, 1.7) {$\bullet$};
			\draw (c.center) to[bend right] node[above] {1}  (a.center);
			\draw (c.center) to[bend right] node[below right] {2} (d.center);
			\draw (c.center) to[bend right] node[above left] {3} (b.center);
		\end{scope}
	\end{tikzpicture}
	\caption{}
 	\end{subfigure}
	\begin{subfigure}[c]{0.24\textwidth}   
	\begin{tikzpicture}[font=\tiny, inner sep =.5]
		\draw[rounded corners] (0,1) -- (0,0) -- (3,0) -- (3,2) --(0,2) --(0,1);
		\begin{scope}[blue]
			\node[label=below:$b$] (b) at (1,0) {$\bullet$};
			\node (c) at (1.5,2) {$\bullet$};
			\node (e) at (3,1) {$\bullet$};
			\draw (c.center) to[bend right] node[above left] {3} (b.center);
			\draw (e.center) to[bend right] node[above left] {4} (b.center);
		\end{scope}
	\end{tikzpicture}
 	\caption{}
  	\end{subfigure}
	\begin{subfigure}[c]{0.24\textwidth}   
	\begin{tikzpicture}[font=\tiny, inner sep =.5]
		\draw[rounded corners] (0,1) -- (0,0) -- (3,0) -- (3,2) --(0,2) --(0,1);
		\begin{scope}[blue]
			\node[label=above:$d$] (d) at (1.5,2) {$\bullet$};
			\node (c) at (0,.5) {$\bullet$};
			\node (e) at (2,0) {$\bullet$};
			\draw (c.center) to[bend right] node[below right] {2} (d.center);
			\draw (e.center) to[bend right] node[above right] {5} (d.center);
		\end{scope}
	\end{tikzpicture}
	\caption{}
 	\end{subfigure}
	\begin{subfigure}[c]{0.24\textwidth}   
	\begin{tikzpicture}[font=\tiny, inner sep =.5]
		\draw[rounded corners] (0,1) -- (0,0) -- (3,0) -- (3,2) --(0,2) --(0,1);
		\begin{scope}[blue]
			\node (b) at (0,.5) {$\bullet$};
			\node (d) at (1,2) {$\bullet$};
			\node[label=below:$e$] (e) at (1.5,0) {$\bullet$};
			\node[label=right:$f$] (f) at (3,1.5) {$\bullet$};
			\draw (e.center) to[bend right] node[above right] {4} (b.center);
			\draw (e.center) to[bend right] node[above right] {5} (d.center);
			\draw (e.center) to[bend right] node[below right] {6} (f.center);
		\end{scope}
	\end{tikzpicture}
		\caption{}
 	\end{subfigure}
 	\caption{The result of cutting up the surface in \Cref{fig:relation_dissection} along the arcs in $R^*_{\overline{Q}}$. For the reader's convenience, we include labels of the original marked points. }
    \label{fig:cutting_example_running}
\end{figure}
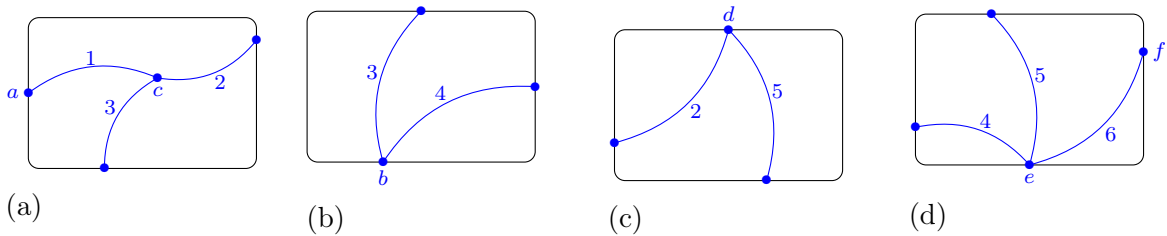

Our goal will be to compare Zembyk's algorithm applied to a locally gentle pair $\overline{Q} = (Q,Z)$ with the effect of splitting $(S_{\overline{Q}}, V_{\overline{Q}}\cup V^*_{\overline{Q}}, D_{\overline{Q}}, R^*_{\overline{Q}})$ with respect to $R^*_{\overline{Q}}$. We check first that taking the levee of $(Q,Z)$ at one relational vertex $v$ is consistent with splitting $(S_{\overline{Q}}, V_{\overline{Q}}\cup V^*_{\overline{Q}}, D_{\overline{Q}}, R^*_{\overline{Q}})$ at $\tau_v^*$.

\begin{lem}\label{lem:SplitAtOneEquivToZembyk}
Let $\overline{Q} = (Q,Z)$ be a locally gentle pair and let $v$ be a relational vertex in $\overline{Q}$. Let $\overline{Q}^v = (Q^v,Z^v)$ denote the levee of $(Q,V)$ at $v$. Then, the split of $(S_{\overline{Q}},V_{\overline{Q}} \cup V_{\overline{Q}}^*, D_{\overline{Q}}, R_{\overline{Q}}^*)$ with respect to $\tau_v^*$ is equivalent to $(S_{\overline{Q^v}},V_{\overline{Q^v}} \cup V_{\overline{Q^v}}^*, D_{\overline{Q^v}}, R_{\overline{Q^v}}^*)$, the surface  corresponding $\overline{Q}^v$, up to homotopy.
\end{lem}

\begin{proof}
We only show the proof in the case where $v$ is a quadbutary vertex, as the proof in the other cases is similar, but simpler.
If $v$ is a quadbutary vertex, then the local configuration of $\overline{Q}$ near $v$ looks like 
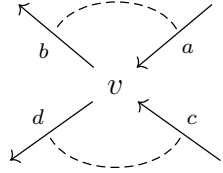
\[
\begin{tikzcd}
{} & {} & {}\ar[dl, "a"name=alp]\\
{} & v\ar[ul, "b"name=bet]\ar[dl, "d"name=betp, swap] & {}\\
{} & {} & {}\ar[ul,  "c"name=alpp, swap]
\ar[-, shorten <= 3pt, shorten >= 3pt, from=alp, to=bet, bend right = 60, dashed]
\ar[-, shorten <= 3pt, shorten >= 3pt, from=alpp, to=betp, bend left = 60, dashed]
\end{tikzcd} 
\]

and in the levee we see 
\[
\begin{tikzcd}[column sep = 1.9em]
{} & v(\sharp) \ar[l, swap, "b^{v}"name=bet] & {}\ar[l, swap,  "c^{v}"name=alpp]\\
{} & v(\flat) \ar[l, swap, "d^{v}"name=betp] & {}\ar[l, swap,  "a^{v}"name=alp]
\end{tikzcd} 
\]
 
Similarly, $D_{\overline{Q}}$ and $R^*_{\overline{Q}}$ near $\tau_v$ and $\tau_v^*$ looks as the following. 

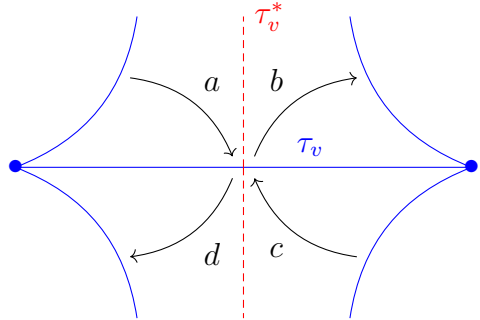
\begin{center}
\begin{tikzpicture}[scale=2]
	\begin{scope}[blue]
	\node (lhs) at (0,0) {$\bullet$};
	\node (rhs) at (3,0) {$\bullet$};
	\draw (lhs.center) to node (v) {} node[above left, pos=0.7] {$\tau_v$} (rhs.center);
	\draw (lhs.center) to[bend right] node[pos=0.7] (1) {} +(0.8,1);
	\draw (lhs.center) to[bend left] node[pos=0.7]  (2) {} +(0.8,-1);
	\draw (rhs.center) to[bend left] node[pos=0.7]  (3) {} +(-0.8,1);
	\draw (rhs.center) to[bend right] node[pos=0.7]  (4) {} +(-0.8,-1);
	\end{scope}

	\draw[<-] (v.north west) to[bend right] node[above right] {$a$} (1);
	\draw[<-] (2) to[bend right] node[below right] {$d$} (v.south west);
	\draw[<-] (3) to[bend right] node[above left] {$b$} (v.north east);
	\draw[<-] (v.south east) to[bend right] node[below left] {$c$} (4);
 	\draw[red,densely dashed]  (v.center) to node[right, pos=]{$\tau_v^*$}  +(0,1cm);
 	\draw[red,densely dashed]  (v.center) to +(0,-1cm);
\end{tikzpicture}
\end{center}

The effect of splitting $(S_{\overline{Q}},V_{\overline{Q}} \cup V^*_{\overline{Q}}, D_{\overline{Q}}, R^*_{\overline{Q}})$ with respect to $\tau_v^*$ near $\tau_v$ looks as the following, where we use superscripts of $i$ and $j$ to  denote which component each piece is in. Note that it is possible that $i = j$.

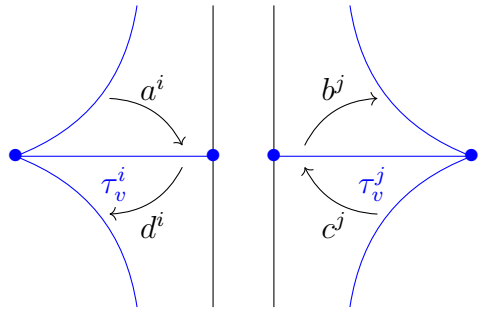
\begin{center}
\begin{tikzpicture}[scale=2]
	\draw (1.3, -1) -- (1.3,1);
	\draw (1.7, -1) -- (1.7,1);
	
	\begin{scope}[blue]
	\node (lhs) at (0,0) {$\bullet$};
	\node (ml) at (1.3,0) {$\bullet$};
	\node (mr) at (1.7,0) {$\bullet$};
	\node (rhs) at (3,0) {$\bullet$};
	\draw (lhs.center) to node[pos=0.9] (v1) {} node[below, pos=0.5] {$\tau_v^i$} (ml.center);
	\draw (mr.center)  to node[pos=0.1] (v2) {} node[below, pos=0.5] {$\tau_v^j$} (rhs.center);
	\draw (lhs.center) to[bend right] node (1) {} +(0.8,1);
	\draw (lhs.center) to[bend left] node (2) {} +(0.8,-1);
	\draw (rhs.center) to[bend left] node (3) {} +(-0.8,1);
	\draw (rhs.center) to[bend right] node (4) {} +(-0.8,-1);
	\end{scope}

	\draw[<-] (v1.north west) to[bend right] node[above] {$a^i$} (1);
	\draw[<-] (2) to[bend right] node[below] {$d^i$} (v1.south west);
	\draw[<-] (3) to[bend right] node[above] {$b^j$} (v2.north east);
	\draw[<-] (v2.south east) to[bend right] node[below] {$c^j$} (4);
\end{tikzpicture}
\end{center}

One can see that the effect of the levee of $\overline{Q}$ locally at $v$ matches the effect of the split locally at $\tau_v$. Since the levee does not affect any other vertices or arrows, and since $\tau_v^*$ does not cross any other arcs in $D_{\overline{Q}}$, we can conclude that these agree for the whole quiver and surface.

\end{proof}

The following is clear since any pair of distinct arcs from $D^*$ do not cross. This is in fact a geometric version of \cref{lem:ZembykOrderDoesntMatter}. 

\begin{lem}\label{lem:GeometricZembykOrderDoesntMatter}
Given a surface $(S, V \cup V^*)$ with a pair of dual dissections $D,D^*$, for any pair of  arcs $\tau_v^*, \tau_w^* \in D^*$, splitting the surface at $\tau_v^*$ and then $\tau_w^*$ results in the same set of dissected surfaces as splitting in the opposite order i.e. splitting the surface at $\tau_w^*$ and then $\tau_v^*$. 
\end{lem}

Since we use the correspondence between gentle algebras and cellular dissections from \cite{Palu-Pilaud-Plamondon-non-kissing-non-crossing}, we verify that the split dissections remain cellular. 

\begin{lem}
Given a surface $(S,V\cup V^*)$ with a pair of dual, cellular dissections $D$ and $D^*$, every dissection $D_i$  in the split with respect to $B \subseteq D^*$ is a cellular dissection of $(S_i,V_i)$.
\end{lem}

\begin{proof}
We show that splitting at any arc $\tau^*$ does not create a non-cellular component. If at least one of the endpoints of $\tau^*$ is not a puncture, then splitting at $\tau^*$ does not create a boundary component that is not connected to an existing boundary component. However, if both endpoints of $\tau^*$ are punctures, as below, splitting at $\tau^*$ creates a new boundary component.

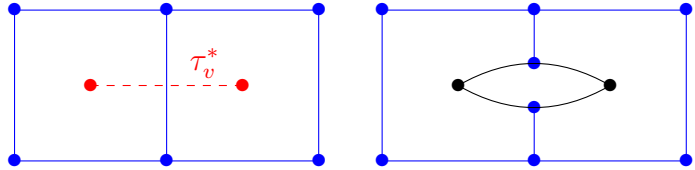
\begin{center}
\begin{tabular}{cc}
    \begin{tikzpicture}
   \draw[blue] (0,0) -- (2,0) -- (2,2) -- (0,2) -- (0,0);
   \draw[blue] (2,0) -- (4,0) -- (4,2) -- (2,2);
   \node[blue] at (0,0) {$\bullet$};
   \node[blue] at (2,0) {$\bullet$};
   \node[blue] at (2,2) {$\bullet$};
   \node[blue] at (0,2) {$\bullet$};
   \node[blue] at (4,2) {$\bullet$};
    \node[blue] at (4,0) {$\bullet$};
    \draw[dashed,red] (1,1) -- (3,1);
    \node[red] at (1,1) {$\bullet$};
    \node[red] at (3,1){$\bullet$};
    \node[red] at (2.5,1.3){$\tau_v^*$};
   \end{tikzpicture}
   &
       \begin{tikzpicture}
   \draw[blue] (2,2) -- (0,2) --  (0,0) -- (2,0);
   \draw[blue] (2,0) -- (4,0) -- (4,2) -- (2,2);
   \node[blue] at (0,0) {$\bullet$};
   \node[blue] at (2,0) {$\bullet$};
   \node[blue] at (2,2) {$\bullet$};
   \node[blue] at (0,2) {$\bullet$};
   \node[blue] at (4,2) {$\bullet$};
    \node[blue] at (4,0) {$\bullet$};
    \node[blue] at (2,0.7){$\bullet$};
    \node[blue] at (2,1.28){$\bullet$};
    \draw[blue] (2,0) -- (2,0.7);
    \draw[blue] (2,2) -- (2,1.28);
    \node[] at (1,1) {$\bullet$};
    \node[] at (3,1){$\bullet$};
    \draw (1,1) to [out = 30, in = 150] (3,1);
    \draw (1,1) to [out = -30, in = 210] (3,1);
   \end{tikzpicture}\\
\end{tabular}
\end{center}

These two punctures are necessarily in two different faces in $\mathcal{F}(D)$, and when we split at $\tau^*$ this will simply add boundary edges to this face, as depicted below. In particular, each such face will remain contractible after the split, proving that the new dissections $D_i$ in each component are cellular. 

\end{proof}

Now, we show that splitting $(S_{\overline{Q}}, V_{\overline{Q}} \cup V_{\overline{Q}}^*, D_{\overline{Q}}^*, R^*_{\overline{Q}})$ with respect to $R^*_{\overline{Q}}$ produces a set of dissections which correspond to the connected components of $Q^{\zembyk}(Z)$. 

\begin{thm}\label{thm:SplitVsZembyk}
Let $\overline{Q}$ be a locally gentle pair with corresponding dissection $D_{\overline{Q}}$ of $(S_{\overline{Q}}, V_{\overline{Q}})$  Let $(S_i,V_{D_i} \cup ,D_i)$ be the components of the split of $(S_{\overline{Q}}, V_{\overline{Q}} \cup V_{\overline{Q}}^*, D_{\overline{Q}}^*, R^*_{\overline{Q}})$ along $R^*_{\overline{Q}}$. Then, the quivers $Q_{D_i}$ associated to these connected components exactly correspond to the connected components of $Q^{\zembyk}(Z)$.
\end{thm}

\begin{proof}
By Lemma \ref{lem:SplitAtOneEquivToZembyk}, splitting at one arc $\tau_v^*$ associated to a relational vertex $v$ is equivalent to taking the levee at $v$. Moreover, by Lemma \ref{lem:GeometricZembykOrderDoesntMatter}, there is a well-defined way to split at multiple arcs, just as by Lemma \ref{lem:ZembykOrderDoesntMatter} we can take the levee simultaneously at a set of vertices. Therefore, it is equivalent to split at all $\tau_v^*$ for $v \in R^*_{\overline{Q}}$ or perform the Zembyk excision of the pair $\overline{Q}$. 
\end{proof}

Since the result of applying Zembyk's algorithm to a locally gentle pair $\overline{Q}$ is a union of quivers of type $A$, type $\tilde{A}$, or cycles without relations, we can also describe the output of performing this split at $R^*_{\overline{Q}}$.

\begin{cor}
Let $(S_{\overline{Q}}, V_{\overline{Q}},D_{\overline{Q}}, R^*_{\overline{Q}})$ be a dissected surface corresponding to a locally gentle pair $\overline{Q}$, and where $R^*_{\overline{Q}} \subseteq D^*_{\overline{Q}}$ is the subset of the dual dissection associated to relational vertices. Then, each connected component $(S_i,M_i,D_i)$ of the split of  $(S_{\overline{Q}}, V_{\overline{Q}} \cup V_{\overline{Q}}^*, D_{\overline{Q}}^*, R^*_{\overline{Q}})$ along $R^*_{\overline{Q}}$ is either a polygon, an annulus, or a once-punctured disk, and for each $i$, there are no internal faces in $\mathcal{F}(D_i)$. 
\end{cor}

We can also prove this corollary independent of the connection to representation theory. After the splitting, all faces of $S_i \backslash D_i$ will be cellular and will be bounded by at most two arcs from $D_i$. The list of options in the corollary correspond to the only ways to glue such faces together to form a connected surface with dissection.

We remark that a related, though different, notion of cutting triangulated surfaces was introduced in \cite{David-Roesler-Schiffler-algebras-from-surfaces-without-punctures} to study a certain class of gentle algebras called ``surface algebras.''

\section{Semilinearity}
\label{sec-semilinearity}
Our results extend the reach of the current literature by allowing \emph{semilinear} representations. We begin by recalling the more general semilinear context. 

\begin{setup}
    Throughout \S\ref{sec-semilinearity} let $K$ be a (associative, unital, but possibly noncommutative) division ring, and write $\Aut(K)$ for the group of ring automorphisms of $K$. 
\end{setup}

We continue with the setup used in  \cite[\S 2.1]{BenTenCraBoe}. 

\begin{defn}\label{defn:rings_over_division_rings}
Recall that a $K$-\emph{ring} is a ring $A$ equipped with a ring homomorphism $K\to A$. Equivalently $A$ is a $K$-bimodule equipped with a $K$-bilinear ring multiplication. Explicitly, a $K$-bimodule $A$ is a $K$-ring if, for any $\lambda \in K $ and for any $a,b\in A$, we have 
\[
\begin{array}{cccc}
 \lambda (ab)=(\lambda a) b,   & (ab)\lambda=a(b\lambda),  & a(\lambda b)=(a\lambda )b,   & \lambda 1=1\lambda.  
\end{array}
\]   
\end{defn}

\begin{defn} \label{defn:semilinear_maps}
Fix $\sigma\in\Aut(K)$. 
For a left (respectively, right) $K$-module $V$ we denote by ${}_{\sigma}V$ (respectively, $V_{\sigma}$) its \emph{left} (respectively, \emph{right}) \emph{twist by} $\sigma$, which is a left (respectively, right) $K$-module consisting of the elements of $V$ subject to the action of $K$ via restriction by $\sigma$. 
So the twisted action of $\lambda\in K$ on $v\in V$ is given by $\sigma(\lambda)v$ (respectively, $v\sigma(\lambda)$). 

A function $\varphi\colon V\to W$ between left $K$-modules is said to be $\sigma$-\emph{semilinear} provided $\varphi(\lambda v)=\sigma(\lambda)\varphi(v)$ for all $\lambda\in K$ and $v\in V$. Hence $\varphi$ is $\sigma$-semilinear if it induces a $K$-linear map from $V$ to the left twist of $W$ by $\sigma$.  
\end{defn}

\begin{example}\label{example:special_cases_semilinear}
The following example discusses several special cases, some of which we will revisit later.
\begin{enumerate}
    \item If $K$ is commutative then $K$-algebras are $K$-rings where the image of $K$ is central. 
    \item If $K$ is a division ring then any $K$-ring contains $K$ as a subring, since the ring homomorphism $K\to A$ must have been injective. 
    \item If $\sigma$ is the identity automorphism of $K$ then any $K$-module is equal to its twist by $\sigma$, and hence $\sigma$-semilinear maps are the same thing as $K$-linear maps. 
\end{enumerate}
\end{example}
Even when $K$ is commutative (in other words, when $K$ is a field) there are interesting examples of $K$-rings which are not $K$-algebras. Likewise there are interesting examples of $\sigma$-semilinear maps which are not linear maps. 

\begin{example}
    \label{examples-of-automorphisms}
    Recall that for a division ring $K$ and an automorphism $\sigma$ of $K$, Ore \cite{Ore-noncopmm-polynomials} gave 
 a division algorithm for the ring $K[x;\sigma]$ of skew polynomials. 
\begin{enumerate}
    \item Prototypical examples of $\BC$-rings which are not $\BC$-algebras are the division ring $\BH=\BR\oplus i\BR\oplus j\BR\oplus k\BR$ of quaternions, and the ring of skew-polynomials $\BC[x;-]$ with respect to complex conjugation. 
    Note the isomorphism $\BC[x;-]/\langle x^{2}+1\rangle\cong \BH$. 
    \item Let $\BF_{p^{n}}$ be the field with $p^{n}$ elements for some prime integer $p>0$, and let $\varphi$ be the Frobenius automorphism of $\BF_{p^{n}}$ given by $f\mapsto f^{p}$. 
   Note that working with rings of the form $\BF_{p^{n}}[x;\varphi]/\langle x^{4}+x^{2}+1\rangle$ gives a useful setting for coding theory over finite fields; see for example \cite[Example 3]{Boucher-Ulmer-coding-with-skew-polynomials}.
   \item One can define spatial rotation in $\BH$ by means of a conjugation automorphism. 
   Namely let $r=\cos (\alpha/2)+\sin(\alpha/2)(i+j+k)$ for some $\alpha\in[0,2\pi)$. Then the automorphism of $\BH$ given by $h\mapsto rhr^{-1}$ sends a quarternion to its rotation by $\alpha$ with respect to the real line. 
   This map defines an $\BR$-linear automorphism of $\BH$. 
   The restriction of this automorphism to subspaces of \emph{timelike} quaternions has been studied, for example, by Ergin and \"{O}zdemir \cite[Theorem 5]{Ozdemir-Ergin-rotations-with-unit-timelike}. 
\end{enumerate}
\end{example}

\subsection{Semilinear path algebras}
\label{sec-semilinear-path-algebras}
We recall a construction from \cite[\S 2.1]{BenTenCraBoe} generalising the usual notion of a path algebra over a field to the context of semilinear maps.

\begin{setup}
 In \S\ref{sec-semilinear-path-algebras} let $Q$ be a finite quiver and $\boldsymbol{\sigma}\colon Q_{1}\to \Aut(K)$ be a function sending each $a\in Q_{1}$ to some $\sigma_{a}\in\Aut(K)$.   
\end{setup}

\begin{defn}\label{defn:semilinear_path_algebra}
The \emph{semilinear path algebra} $K_{\boldsymbol{\sigma}} Q$ is the $K$-ring generated by the trivial paths and arrows, subject to the relations 
\[
\begin{array}{cccl}
  \sum_{v\in Q_0} e_{v} = 1, & e_{v}e_{v}=e_{v}, & e_{u} e_{v} = 0, &  (u,v\in Q_{0}\text{ with }u\neq v).\\[-1em]\\
  e_{v} \lambda = \lambda e_{v}, & e_{h(a)} a = a = a e_{t(a)}, & a \lambda = \sigma_{a}(\lambda) a, &  (v\in Q_{0}, a\in Q_{1}, \lambda \in K).
\end{array}
\]
\end{defn}

\begin{rem}
Taking $K$ to be commutative, and hence a field, the usual path algebra $KQ$ is defined using the first 6 equalities above, together with the simplified equation, $\lambda a=a\lambda$. Hence in this setting $KQ=K_{\boldsymbol{\sigma}}Q$ where each $\sigma_{a}$ is the identity on $K$.
\end{rem}

\begin{example}\label{example:semilinear_loop_example}
If $Q$ is a single loop, then $K_{\boldsymbol{\sigma}} Q$ is isomorphic to a skew-polynomial ring $K[x;\sigma]$, where $\sigma$ is the automorphism of $K$ associated to this loop.
\end{example}

We next give an example of a semilinear paths algebra of classical Dynkin type $\BA$.

\begin{example}
Let $K$ and $\sigma_{\alpha}=\rho,\sigma_{\beta}=\tau\in\Aut(K)$  be  arbitrary where 
$Q$ is the quiver
\[
 \begin{tikzpicture}[xscale=2]
	\node (1) at (1,0) {1};
	\node (2) at (2,0) {2};
	\node (3) at (3,0) {3};
	\draw[->] (1) to node[above]{$\alpha$}(2);
	\draw[->] (2) to node[above]{$\beta$}(3);
\end{tikzpicture}.\]
The semilinear path algebra $K_{\rho, \tau}Q$ is isomorphic to the $K$-ring  $A$ given by matrices 
$$M = \begin{pmatrix}
    \lambda & 0 & 0 \\
    \mu & \lambda' & 0 \\
    \eta & \mu' & \lambda''
\end{pmatrix}, \text{ where } \mu, \lambda, \lambda', \lambda'' \in 
K, \eta, \mu' \in 
K_{\tau},$$
where $
K_{\tau}$ is the module given by $\varphi\colon  
K_{\tau} \otimes_{
K} 
K \to 
K_{\tau}$ defined by $\eta \otimes \lambda \mapsto \eta \tau(\lambda)$. Multiplication in $A$ respects the $
K_{\tau}$ structure in some of its entries via
$$\begin{pmatrix}
     \lambda & 0 & 0 \\
    \mu & \lambda' & 0 \\
    \eta & \mu' & \lambda''
\end{pmatrix} \begin{pmatrix}
    \kappa & 0 & 0 \\
    \iota & \kappa' & 0 \\
    \theta & \iota' & \kappa''
\end{pmatrix} = \begin{pmatrix}
     \lambda \kappa & 0 & 0 \\
    \mu \kappa+\lambda' \iota & \lambda' s' & 0 \\
    \eta \kappa+\mu'\tau(\iota) + \lambda''\theta & \mu'\kappa'+\lambda''\iota' & \lambda''\kappa''
\end{pmatrix}.$$

The isomorphism is defined by sending 
$$e_1 \mapsto \begin{pmatrix}
    1 & 0 & 0 \\
    0 & 0 & 0 \\
    0 & 0 & 0
\end{pmatrix}; \hspace{1em} e_2 \mapsto \begin{pmatrix}
    0 & 0 & 0 \\
    0 & 1 & 0 \\
    0 & 0 & 0
\end{pmatrix}; \hspace{1em} e_3 \mapsto \begin{pmatrix}
    0 & 0 & 0 \\
    0 & 0 & 0 \\
    0 & 0 & 1
\end{pmatrix};$$
$$\alpha \mapsto \begin{pmatrix}
    0 & 0 & 0 \\
    1 & 0 & 0 \\
    0 & 0 & 0
\end{pmatrix}; \hspace{1em} \beta \mapsto \begin{pmatrix}
    0 & 0 & 0 \\
    0 & 0 & 0 \\
    0 & 1 & 0
\end{pmatrix}.$$

\end{example}

We now recall the definition of a tensor ring in order to use it to state a condition for when a semilinear path is a hereditary $K$-ring. 

\begin{defn}\label{defn:tensor_ring}
Let $R$ be a ring, and let $M$ be an $R$-$R$-bimodule. The \emph{tensor ring} $T_{R}(M)=\bigoplus_{n\geq0} M^{\otimes n}$ of $M$ over $R$ is given by the direct sum of tensor powers $M^{\otimes n}=M\otimes_{R}\dots \otimes_{R}M$ with a multiplication induced by the canonical $R$-balanced maps $M^{\otimes n}\times M^{\otimes m}\to M^{\otimes n+m}$. 
As noted in \cite{Cohn-1995-skew-fields-book},  the tensor ring $T_{R}(M)$ is an $R$-ring. 
Furthermore, if we assume  $S$ to be a central subring of $R$ then provided $ms=sm$ for all $s\in S$ and $m\in M$, we have that the $R$-ring $T_{R}(M)$ is also an $S$-algebra.  
\end{defn}

\begin{rem}
\label{rem-semilinear-path-algebra-is-a-tensor-ring}
    We explain how the semilinear path algebra arises as a tensor ring, as was explained at the beginning of \cite[\S 2.1]{BenTenCraBoe}. 
    Consider the semisimple ring $R=K^{Q_{0}}$ defined as the direct product of $|Q_{0}|$ copies of the  division ring $K$. 
Consider the $R$-$R$-bimodule $M=K^{Q_{1}}$ where the actions of $R$  
on the left and right are respectively, defined by
\[
\begin{array}{ccc}
(\lambda_{v})(\mu_{a})=(\lambda_{h(a)}\mu_{a}), & (\mu_{a})(\eta_{v})=(\mu_{a}\sigma_{a}(\eta_{t(a)})), & ((\lambda_{v}),(\eta_{v})\in R,\,(\mu_{a})\in M).
\end{array}
\]
Consider the map $K_{\boldsymbol{\sigma}}Q\to T_{R}(M)$ sending a path $p= a_{1}\dots  a_{n}$ to the pure tensor $m_{1}\otimes\dots \otimes m_{n}\in M^{\otimes n}$ where we define each $m_{i}=(\delta_{i,a}\colon a\in Q_{1})$ by defining  $\delta_{i,a}\in K$ by $\delta_{i,a}=1$ when $a=a_{i}$, and $\delta_{i,a}=0$ otherwise. 
Since the paths in $Q$ give a (left or right) $K$-basis for $K_{\boldsymbol{\sigma}}Q$, this map is an isomorphism of $K$-rings. 
\end{rem}

\begin{prop}
\label{prop:acylic-semilinear-implies-hereditary-ring}
If $Q$ is acyclic  then $K_{\boldsymbol{\sigma}}Q$ is a hereditary $K$-ring. 
\end{prop}

\begin{proof}
We recall and keep the notation from \Cref{rem-semilinear-path-algebra-is-a-tensor-ring}. 
Since $Q$ is acyclic there is an upper bound on the length of a path in $Q$. 
Altogether this means $M^{\otimes n}=0$ for $n\gg 0$. 
That $K_{\boldsymbol{\sigma}}Q$ is a hereditary ring  follows immediately from a result of Jans and Nakayama \cite[Proposition 4]{Jans-Nakayama-On-the-Dimension-of-Modules-and-Algebras}, noting that $R$ is semisimple. 
\end{proof}

\begin{lem}
\label{lem:primitive-idempotents-in-f-d-semilinear-quotients}
(c.f. \cite[Lemmas 2.4 and 2.10]{AssemSimsonSkowronski-Vol1}.)
    Let $A$ be the ideal in $K_{\boldsymbol{\sigma}}Q$ generated by the arrows in $Q$, and let $\Lambda=K_{\boldsymbol{\sigma}}Q/J$ where  $J$ is an ideal in $K_{\boldsymbol{\sigma}}Q$ with $J\subseteq A^{2}$ and $A^{d}\subseteq J$ for some $d>1$. 
   Then the cosets $\overline{e}_{v}=e_{v}+J\in\Lambda$  ($v\in Q_{0}$) define a complete and orthogonal set of primitive idempotents. 
   
Furthermore $\rad(\Lambda)=A/J$, the two-sided ideal generated by the arrows in $Q$.   
\end{lem}

\begin{proof}
   The elements  $\overline{e}_{v}$ define a complete set of orthogonal idempotents in $\Lambda$, since the elements $e_{v}$ do so in $K_{\boldsymbol{\sigma}}Q$. 
    We fix $v\in Q_{0}$ and an idempotent $f$ of the ring $\Omega=\overline{e}_{v}\Lambda \overline{e}_{v}$, and we claim $f=0$ or $f=1$ in $\Omega$. 
    
By construction  $f=\mu e_{v}+x+J$ for some $\mu\in K$ and $x\in e_{v}Ae_{v}$.
Since  $f^{2}=f$ and hence $f^2 - f = 0$, we have $(\mu^{2}-\mu)e_{v}+(\mu -1)x +x \mu+x^{2}\in J$. 
Since $J\subseteq A^{2}\subseteq A$ and $(\mu -1)x ,x \mu,x^{2}\in A$  this means  $\mu(\mu-1)e_{v}\in A$. 
Since the paths give a $K$-basis for $K_{\boldsymbol{\sigma}}Q$, and since those of non-zero length give a $K$-basis for $A$, we have $\mu=0$ or $\mu=1$. 
Let $\overline{x}=x+J\in \Omega$. 
So either ($\mu=0$ and $\overline{x}$ is idempotent) or ($\mu=1$ and $-\overline{x}$ is idempotent). 
 
Since $A^{d}\subseteq J$ we have $x^{d}\in J$, and so $\overline{x}$ and $-\overline{x}$ are both nilpotent elements $\Omega$. 
Hence whether $\mu=0$ or $\mu=1$, in either case we have $\overline{x}=0$  and so $x\in e_{v}J e_{v}$, meaning $f=\mu e_{v}+J$. 
Thus, either ($\mu=0$ and hence $f=0$) or ($\mu=1$ and hence $f=1$).  

Let $I=A/J$. 
We now claim $\rad(\Lambda)=I$. Note that the paths of length at most $d-1$ must span $\Lambda$ as a $K$-module. 
Hence  $\Lambda$ is artinian. 
By assumption we have that $A^{d}\subseteq J$, meaning $A/J$ is nilpotent, and so $I = A/J\subseteq \rad(\Lambda)$; see for example \cite[Theorem 4.12]{Lam-first-course-noncomm-rings}. 
Note also that  $\Lambda/I\cong K_{\boldsymbol{\sigma}}Q/A\cong K^{n}$ where $n$ is the number of vertices in $Q$. 
Hence $\rad(\Lambda/I)=0$ (in other words, the ring $\Lambda/I$ is \emph{J}-\emph{semisimple}), and hence $I = A/J\supseteq \rad(\Lambda)$; see for example \cite[p. 68, Exercise 11]{Lam-first-course-noncomm-rings}. 
\end{proof}

\subsection{Semilinear gentle algebras and nodal $K$-rings}
\label{semilinear-gentle-implies-nodal}

\begin{setup}
    Throughout \Cref{semilinear-gentle-implies-nodal} let $\boldsymbol{\sigma}=(\sigma_{a})$ be a function $Q_{1}\to \Aut(K)$. 
\end{setup}

\begin{defn}  
By a \emph{semilinear locally gentle algebra over $K$} we mean a $K$-ring of the form $K_{\boldsymbol{\sigma}} Q/I$ where $I=\langle Z\rangle$ and where $(Q,Z)$ be a locally gentle pair. 

Let $\Lambda=K_{\boldsymbol{\sigma}} Q/I$ be a semilinear locally gentle algebra. 
Note that the paths in $Q$ that do not factor through a relation in $Z$ define a $K$-basis for $\Lambda$ by \cite[Theorem 2.22]{BenTenCraBoe}. 

In case this set of paths is finite, or equivalently, in case $\Lambda$ is finite-dimensional as a $K$-module, we say that $\Lambda$ is a \emph{semilinear gentle algebra}. 
\end{defn}

\begin{rem}
\label{remark-about-zembyks-excision-algebraic}
    We are now ready to supplement \Cref{remark-about-zembyks-excision} by explaining the main result of \cite{Zembyk-Skewed-Gentle-A} using the language of classical path algebras.

    Let $K$ be a field, $Q$ be a quiver, $Z$ be a set of paths in $Q$, and $KQ$ be the path algebra. 
    Note $KQ=K_{\boldsymbol{\sigma}}Q$ where each automorphism $\sigma_{a}$ ($a\in Q_{1}$) is taken to be the identity. 
    To say that the pair $(Q,Z)$ is gentle is the same as saying that the quotient $\Lambda=KQ/\langle Z\rangle $ is a (finite-dimensional) \emph{gentle algebra} in the usual sense of Assem  and Skowro\'{n}ski 
    \cite{assem-skowronski-gentle}. 
    
      In \Cref{remark-about-zembyks-excision} we noted  that when $(Q,Z)$ is gentle, the path algebra $\Gamma=KQ^{\zembyk}(Z)$ is a direct product of finitely many copies of path algebras of type $\BA$ and  $\tilde{\BA}$. 
    The main result in work of Zembyk \cite[p. 649, Theorem]{Zembyk-Skewed-Gentle-A} also exhibits an injective $K$-algebra homomorphism $\Lambda\to \Gamma$ such that $\rad(\Lambda)=\rad(\Gamma)$ and such that for any simple $\Lambda$-module $U$ the $\Lambda$-module $\Gamma\otimes_{\Lambda}U$ has length at most $2$. 
    In this way $\Lambda$ is a \emph{nodal algebra of type} $\BA$. 
\end{rem}

In light of \Cref{remark-about-zembyks-excision-algebraic} we introduce the following definition. 

\begin{defn}\cite[Definition 1]{Zembyk-Skewed-Gentle-A}
\label{def:semilinear-nodal}
    A finite-dimensional $K$-ring $\Lambda$ is said to be \emph{nodal} if it is a subring of a finite-dimensional  hereditary $K$-ring $\Gamma$ such that the following  holds. 
\begin{enumerate}
    \item The ring embeddings $K\to \Lambda$ and $\Lambda\to\Gamma$ compose to the ring embedding $K\to \Gamma$. 
    \item We have $\rad(\Lambda)=\rad (\Gamma)$.
    \item For any simple left  $\Lambda$-module $U$ the left $\Lambda$-module $\Gamma\otimes_{\Lambda}U$ has length at most $2$.  
\end{enumerate}
In this case we say $\Lambda$ is \emph{connected with} $\Gamma$. 
\end{defn}

\begin{setup}
\label{gentle-implies-nodal-setup}
    For what remains in \Cref{semilinear-gentle-implies-nodal} we let $\overline{Q} = (Q,Z)$ be a gentle pair, $I=\langle Z\rangle$ and  $\Lambda=K_{\boldsymbol{\sigma}}Q/I$, the associated (finite-dimensional) semilinear gentle algebra. 
    We also let  $Q'=Q^{\zembyk}(Z)$. 
    Recall $Q'$ was described iteratively in \Cref{defn-zembyk-quiver-scissors}.

    We introduce notation for $Q'$ for the remainder of \Cref{semilinear-gentle-implies-nodal}. We let $Q_{0}^{r}$ denote the set of relational vertices in $Q$. 
Then we write
\[
\begin{array}{cc}
    Q_{0}'=\{v'\in Q_{0}\setminus Q_{0}^{r}\}\cup \{u(\sharp),u(\flat)\colon u\in Q_{0}^{r}\}, 
    &
    Q_{1}'=\{a'\colon a\in Q_{1}\}. 
\end{array}
\]
We write $h',t'\colon Q'_{1}\to Q'_{0}$ for the head and tail functions, which satisfy $\{ h'(a'),t'(b')\}=\{v(\sharp),v(\flat)
\}$ for any  $v\in Q_{0}^{r}$ of the form $h(a)=t(b)$ where $ba\in Z$ for some $a,b\in Q_{1}$. 

We then let $\Gamma=K_{\boldsymbol{\sigma}'}Q'$ where we define the function $\boldsymbol{\sigma}'\colon Q'_{1}\to \Aut(K)$ by $a'\mapsto \sigma_{a}$. 
\end{setup}

\begin{lem}
\label{lem:acyclic-quiver}
    The quiver $Q'=Q^{\zembyk}(Z)$ is acyclic.
\end{lem}

\begin{proof}
    For a contradiction assume there is an oriented cycle $c'=c'_{1}\dots c'_{t}$ ($c'_{i}\in Q_{1}'$)  in $Q'$ corresponding to an oriented cycle $c=c_{1}\dots c_{t}$ ($c_{i}\in Q_{1}$) in $Q$. 
Since $\Lambda $ is finite-dimensional we must have that $c^{n}\in I$ for some $n>0$. 

By \cite[Theorem 2.22]{BenTenCraBoe} this means that $c_{j}c_{j+1}\in Z$ for some $j=1,\dots, t$ where $c_{t+1}:=c_{1}$. 
Let $v\in Q_{0}^{r}$ be the vertex $t(c_{j})=h(c_{j+1})$. 
So we must have  $t'(c'_{j})=v(\sharp)$ and $h'(c'_{j+1})=v(\flat)$, but since $v(\sharp)\neq v(\flat)$ by construction, this  contradicts that $c'$ is a cycle. 
\end{proof}

\begin{lem}
\label{lem:condition-1-of-nodality}
    There is an injective ring homomorphism $\Delta \colon \Lambda\to \Gamma$ which satisfies condition (1) of \Cref{def:semilinear-nodal}. 
\end{lem}

\begin{proof}
 Let $S$ be the free $K$-ring generated by the elements $e_{v}$ and $a$ with $(v,a)\in Q_{0}\times Q_{1}$, and let $T$ be the free $K$-ring  generated by the elements $e_{u}$ and $b$ with $(u,b)\in Q_{0}'\times Q_{1}'$. 
 
    We firstly claim there is a map $ K_{\boldsymbol{\sigma}}Q\to K_{\boldsymbol{\sigma}'}Q'$ of $K$-rings. To do so we begin by observing that there is a homomorphism of $K$-rings of the form $S\to T$ defined by $K$-linearly and multiplicatively extending  the assignments
\[
\begin{array}{ccc}
e_{v}\mapsto e_{v'} \quad (v\in Q_{0}\setminus Q_{0}^{r}),
&
e_{u}\mapsto  e_{u(\sharp)}+e_{u(\flat)} \quad (u\in Q_{0}^{r}), 
&
a\mapsto a' 
\quad 
(a\in Q_{1}).
\end{array}
\] 
To define a map $ K_{\boldsymbol{\sigma}}Q\to K_{\boldsymbol{\sigma}'}Q'$ we compose the map $S\to T$ with the canonical quotient map $ T\to K_{\boldsymbol{\sigma}'}Q'$. 
It is straightforward to check that the relations from \Cref{defn:semilinear_path_algebra} are sent to $0$ under this composition, by noting: that $\{u(\sharp),u(\flat)\}\cap\{v(\sharp),v(\flat)\}=\emptyset$ for any distinct $u,v\in Q_{0}^{r}$; that $\{u(\sharp),u(\flat)\}\cap\{v'\}=\emptyset$ for any distinct $u\in Q_{0}^{r}$ and any $v\in Q_{0}\setminus Q_{0}^{r}$; and that precisely one of the elements  $v(\sharp), v(\flat)$ is equal to the head (respectively, tail) of the arrow $a'$ for any arrow $a\in Q_{1}$ with head (respectively, tail) $v$. 

Thus the map $S\to T$ gives rise a map of $K$-rings  $\nabla\colon  K_{\boldsymbol{\sigma}}Q\to K_{\boldsymbol{\sigma}'}Q'$ that we now claim  gives rise to an injective map $\Delta\colon \Lambda\to \Gamma$ of $K$-rings that satisfies condition (1) of \Cref{def:semilinear-nodal}. 
Let $ba\in Z$, and write $v$ for the relational vertex $t(b)=h(a)$. 
By construction and without loss of generality we have $t'(b')=v(\sharp)$ and $h'(a')=v(\flat)$, and so $\nabla(ba)=b'a'=0$ in  $K_{\boldsymbol{\sigma}'}Q'=\Gamma$. Thus $Z\subseteq \ker(\nabla)$. 

We assert the reverse inclusion holds. 
Now suppose $x\in \ker(\nabla)$ and write $\lambda_{p}\in K$ for the coefficient of a given path $p$ in $Q$ arising in the expression for $x$. 
Suppose $p=a_{1}\dots a_{n}$ is a path in $Q$ 
 with $a_{i}\in Q$ and such that $a_{i}a_{i+1}\notin Z$  whenever $i<n$. 
We have that $h(a_{i+1}')=t(a_{i}')$ whenever $i<n$, for otherwise, without loss of generality, we have  $h(a_{j+1}')=v(\sharp)$ and $t(a_{j}')=v(\flat)$ for some $j<n$, in which case $a_{j}a_{j+1}\in Z$, a contradiction. 
Hence $\nabla(p)$ is a path in $Q'$ and this means $\lambda_{p}=0$ since $x\in \ker(\nabla)$. 
Hence $x$ is a $K$-linear combination of paths factoring through elements of $Z$. Thus $Z\supseteq \ker(\nabla)$. 

That condition (1) from \Cref{def:semilinear-nodal} holds is now trivial. 
\end{proof}

\begin{thm}
\label{thm-semilinear-fd-gentle-implies-nodal}
   Any finite-dimensional semilinear gentle algebra $\Lambda=K_{\boldsymbol{\sigma}}Q/\langle Z\rangle  $ is nodal, connected with $\Gamma=K_{\boldsymbol{\sigma}'}Q^{\zembyk}(Z)$. 
\end{thm}

\begin{proof}
Recall the notation from \Cref{gentle-implies-nodal-setup}. 
By \Cref{lem:condition-1-of-nodality} there is a embedding of $K$-rings $\Delta\colon \Lambda\to \Gamma$ satisfying  condition (1) of \Cref{def:semilinear-nodal}. 

By \Cref{lem:acyclic-quiver} the quiver $Q'=Q^{\zembyk}(Z)$ is acyclic, and so a trivial application of \Cref{lem:primitive-idempotents-in-f-d-semilinear-quotients} shows that $\rad(\Gamma)$ is generated by the arrows in $Q'$ and hence the images of the arrows in $Q$. 
Since $Z$ consists of paths of length $2$, and since $\Lambda$ is finite-dimensional, another application of \Cref{lem:primitive-idempotents-in-f-d-semilinear-quotients} shows that  $\rad(\Lambda)$ is generated by the arrows in $Q$. 
Hence condition (2) of \Cref{def:semilinear-nodal} holds. 

Another conclusion of  \Cref{lem:primitive-idempotents-in-f-d-semilinear-quotients}   is that $(e_{v}+I\colon v\in Q_{0})$  defines a complete and orthogonal sequence of primitive idempotents in $\Lambda$.  
Hence by the theory of semiperfect rings the modules $\Lambda e_{v}/Ae_{v}$ constitute a complete list of simple left $\Lambda$-modules; see \cite[Theorem 25.3]{Lam-first-course-noncomm-rings}. 
Thus to verify condition (3) of \Cref{def:semilinear-nodal} holds it suffices to let $U=\Lambda e_{v}/Ae_{v}$ for some $v\in Q_{0}$ and prove that $\Gamma\otimes_{\Lambda} U$ has length at most $2$ as a left $\Lambda$-module. 
By \cite[p. 469, Lemma 1]{Zembyk-structure-of-finite-dim-nodal-algs} it is equivalent to prove that $\overline{\Gamma}\otimes_{\overline{\Lambda}} U$ has length at most $2$ as a left $\overline{\Lambda}$-module where $\overline{\Lambda}=\Lambda/\rad(\Lambda)$ and $\overline{\Gamma}=\Gamma/\rad(\Gamma)$. 

We now claim, and it is sufficient to prove, that $\overline{\Gamma}\otimes_{\overline{\Lambda}} U$ is generated by at most $2$ elements as a $K$-module. 
As we saw above, by \Cref{lem:primitive-idempotents-in-f-d-semilinear-quotients} the $K$-module $\overline{\Gamma}$ is spanned over $K$ by the cosets $e_{w}+\rad(\Gamma)$ with $w\in Q_{0}'$. 
Likewise, as $e_{u}\Lambda e_{v}\subseteq e_{u}Ae_{v}$ for any vertex $u\neq v$, we have that $e_{v}+I$ is a  $K$-basis for $U$. 
Hence the pure tensors defined by
\[
x_{w}=(e_{w}+\rad(\Gamma))\otimes (e_{v}+I)\quad (w\in Q_{0}')
\]
constitute a $K$-basis for $\overline{\Gamma}\otimes_{\overline{\Lambda}}U$. Note that $x_{w}=(y_{w}+\rad(\Gamma))\otimes (e_{v}+I)$ where $y_{w}=e_{w}\Delta(e_{v}+I)$. 
It suffices to show that $y_{w}=0$ for all but at most $2$ values of $w$. 

Firstly suppose $v\in Q_{0}\setminus Q_{0}^{r}$. 
Then $\Delta(e_{v}+I)=e_{v'}$ and so $y_{w}=e_{w}e_{v'}$ meaning  $y_{w}\neq 0$ if and only if $w=v'$. 
Secondly suppose $v\in Q_{0}^{r}$. 
Then $\Delta(e_{v}+I)=e_{v(\sharp)}+e_{v(\flat)}$ and so $y_{w}=e_{w}e_{v(\sharp)}+e_{w}e_{v(\flat)}$, meaning  $y_{w}\neq 0$ if and only if $w=v(\sharp)$ or $w=v(\flat)$. 
\end{proof}

\begin{rem}
\label{rem-conjecture}
    In recent work of the second author and Labardini-Fragoso \cite{bennetttennenhaus2023semilinear} 
    tensor algebras of bimodules defined by species with potential are considered, and shown to be morita equivalent to certain quotients of semilinear path algebras; see \cite[Theorem 1]{bennetttennenhaus2023semilinear}. 
    These quotients are reminiscent of the \emph{skewed}-\emph{gentle} algebras introduced Bessenrodt and  Holm \cite{bessholmskewgentle}, which generalise gentle algebras in the classical sense introduced in \cite{assem-skowronski-gentle}. Note that Zembyk's main result in  \cite{Zembyk-Skewed-Gentle-A} shows that skewed-gentle algebras are nodal. 
    Hence we conjecture that \Cref{thm-semilinear-fd-gentle-implies-nodal} generalises to a class of quotients of semilinear path algebras which includes those considered in \cite{bennetttennenhaus2023semilinear}. 
\end{rem}

\subsection{Semilinear tiling algebras}\label{subsec:SemilinearTilingAlgebra}

\begin{setup}
Throughout Section \ref{subsec:SemilinearTilingAlgebra}, let $\overline{Q} = (Q,Z)$ denote a locally gentle pair. 
\end{setup}

In order to realise a semilinear locally gentle algebra $K_{\boldsymbol{\sigma}}Q/\langle Z \rangle$ by means of a surface, we will use the injection $\mathcal{L}$ from arrows of $\overline{Q}$ to $\mathcal{F}(D_{\overline{Q}} \cup R^*_{\overline{Q}})$ in Corollary \ref{cor:ArrowsToFaces} to include labels to faces. 

\begin{defn}\label{def:SemilinearTiling}
A \emph{labeled tiling} is a tuple $(S,M,D,R^*,\boldsymbol{\ell})$ such that 
\begin{enumerate}
    \item there is a partitioning $M = V \cup 
  V^*$ such that the elements of $V$ and $V^*$ alternate on each boundary component of $S$ and such that there exist dual cellular dissections $D, D^*$ with endpoints in $V$ and $V^*$ respectively;
    \item $R^*  \subseteq D^*$ is such that each $F \in \mathcal{F}(D \cup R^*)$ is a disk with either \begin{itemize}
        \item one side of $F$ coming from from $D$ and all other sides from the boundary or $R^*$ (type 1) or
         \item two consecutive sides of $F$ coming from $D$ and all other sides from the boundary or $R^*$ (type 2); and
    \end{itemize}
    \item $\boldsymbol{\ell}$ is a function from the set of polygons of type 2 to $\Aut(K)$. 
\end{enumerate}
\end{defn}

Note that the map $\mathcal{L}\colon  Q_1 \to \mathcal{F}(D_{\overline{Q}} \cup R^*_{\overline{Q}})$ from Corollary \ref{cor:ArrowsToFaces} can be improved to a bijection when we restrict the image to faces of type 2. 

Given a cellular dissection $D$ of $(S,V)$, we say that the associated locally gentle algebra from the locally gentle pair $\overline{Q}_D = (Q_D,Z_D)$, as described in \S\ref{sec:SurfaceDissection}, is the \emph{tiling algebra} of $(S,M,D)$. Now we use the additional information of a labeled tiling to define a semilinear tiling algebra. 

\begin{defn}\label{def:SemilinearTilingAlgebra}
Let $(S,V \cup V^*,D,R^*,\boldsymbol{\ell})$ be a labeled tiling. The associated \emph{semilinear tiling algebra} $K_{\boldsymbol{\sigma_\ell}}Q_D/I_D$ has the tiling algebra of $(S,V,D)$ as underlying algebra and assigns $\boldsymbol{\ell} \circ \mathcal{L}(a) \in \text{Aut}(K)$ to each $a\in Q_1$.  
\end{defn}

\begin{prop}
Every semilinear locally gentle algebra is a semilinear tiling algebra.
\end{prop}

\begin{proof}
Consider a semilinear locally gentle algebra  $K_{\sigma}Q/I$.  
Theorem 4.10 of \cite{Palu-Pilaud-Plamondon-non-kissing-non-crossing} gives a tiling $(S,M,D)$ such that $KQ/I$ is the tiling algebra $KQ_D/I_D$.  By Lemma \ref{lem:DivideIntoSmallerFaces}, we can add additional arcs $R^*$ from the dissection dual to $D$ such that $D \cup R^*$ satisfies the conditions of Definition \ref{def:SemilinearTiling}. Then, by choosing $\boldsymbol{\ell}(F) = \sigma_{\mathcal{L}^{-1}(F)}$ for every face of type 2, $F \in \mathcal{F}(D \cup R^*)$, we see that the semilinear tiling algebra $K_{\sigma_{\boldsymbol{\ell}}}Q_D/I_D$ from $(S,M,D,R^*,\boldsymbol{\ell})$ is exactly $K_{\boldsymbol{\sigma}}Q/I$. 
\end{proof}

\begin{example}\label{running-example:tiling}
Consider the locally gentle pair $\overline{Q}$ from \cref{running-example:zembyk-excision}. For each $a\in Q_1$, choose $\sigma_a\in\Aut(K)$. 
Figure \ref{fig:semilinear tiling} gives the labeled tilting corresponding to $K_{\boldsymbol{\sigma}}Q/\langle Z\rangle $.

\begin{figure}[H]
    \centering
    \begin{tikzpicture}[inner sep=1, outer sep = 0]
        \draw[rounded corners] (0,1)--(0,5)--(10,5)--(10,0)--(0,0)--(0,1);
        
        \begin{scope}[blue]
        \node[label=left:$a$]  (a) at (0,4) {$\bullet$};
        \node[label=below:$b$] (b) at (2,0) {$\bullet$};
        \node[label=below right:$c$] (c) at (4,3) {$\bullet$};
        \node[label=above:$d$] (d) at (6,5) {$\bullet$};
        \node[label=below:$e$] (e) at (7,0) {$\bullet$};
        \node[label=right:$f$] (f) at (10,3) {$\bullet$};

        \draw (a.center) to[bend left] node[above right] {1} (c.center);
        \draw (d.center) to[bend left] node[above left] {2} (c.center);
        \draw (b.center) to[bend left] node[above left] {3} (c.center);
        \draw (b.center) to[bend left] node[above right] {4} (e.center);
        \draw (d.center) to[bend left] node[below right] {5} (e.center);
        \draw (f.center) to[bend left] node[below right] {6} (e.center);

        \end{scope}
        \begin{scope}[red, densely dashed]
        \node (i) at (3,5) {$\bullet$};
        \node (ii) at (0,1) {$\bullet$};
        \node (iii) at (6,3) {$\bullet$};
        \node (iv) at (4.5,0) {$\bullet$};
        \node (v) at (8,5) {$\bullet$};
        \node (vi) at (10,1) {$\bullet$};
        
        \draw (i.center) to[bend right] (iii.center);
        \draw (ii.center) to[bend right] (iii.center);
        \draw (iv.center) to[bend right] (iii.center);
        \draw (v.center) to[bend right] (iii.center);
        
        \node at (3.2, 4) {$\sigma_\alpha$};
        \node at (4, 2) {$\sigma_\beta$};
        \node at (1.5, 1.2) {$\sigma_\nu$};
        \node at (6, 4) {$\sigma_\delta$};
        \node at (6.5, 1.5) {$\sigma_\epsilon$};
        \node at (8, 2) {$\sigma_\eta$};
        \node at (5, 1.5) {$\sigma_\zeta$};

        \end{scope}
         
        \begin{scope}[lightgray, densely dashed]
        \draw (ii.center) to[bend right] (i.center);
        \draw (v.center) to[bend right] (vi.center);
        \end{scope}
        
    \end{tikzpicture}

    \caption{The labeled tiling corresponding to the semilinear gentle algebra from Example \ref{running-example:zembyk-excision}. The boundary of the surface $S$ is in black. The marked points $M=V\cup V^*$ are in blue (for $V$) and red (for $V*$). The dissection $D$ is in blue, while the partial dual dissection $R^*$ is in red with dashed lines. The face labels provided by $\boldsymbol{\ell}$ are in red.}
    \label{fig:semilinear tiling}
\end{figure}
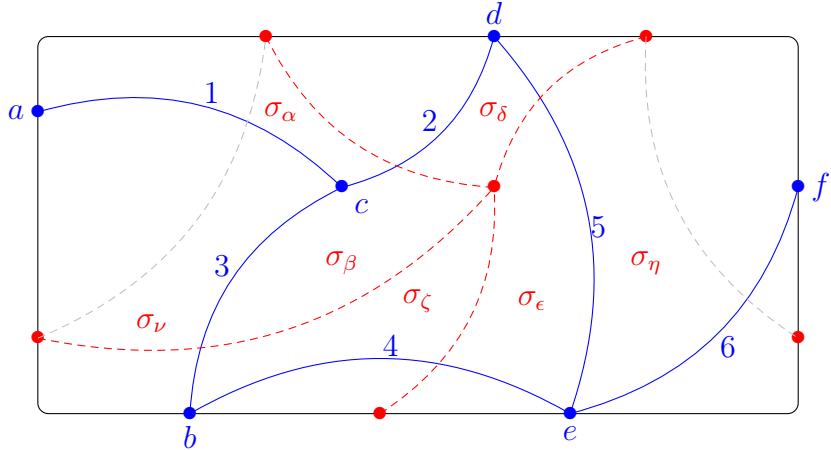

\end{example}

\section{Geometric model for semilinear representations}
\label{sec-geometric-model}
 The first three subsections of this section build to a classification of the finite-dimensional modules of a semilinear locally gentle algebra (see Theorem \ref{BTCB-main-theorem}). This comes from \cite{BenTenCraBoe}, which considers a strictly larger class of algebras. In particular, semilinear locally gentle algebras are semilinear clannish algebras that have no \emph{special loops}, have \emph{quadratic} zero-relations, but have a restriction on the number of zero-relations.

In the final section, we review the geometric model for indecomposable modules of gentle algebras following \cite{Baur-Coelho-Simoes-geometric-model-module-cat-gentle}. We then employ the work of \cite{Palu-Pilaud-Plamondon-non-kissing-non-crossing} to extend this model to include modules of locally gentle algebras. We conclude with a geometric construction which includes semilinearity into this story by using the face labels constructed in \S\ref{subsec:SemilinearTilingAlgebra}.


\subsection{Admissible words}
\label{subsec-admissible-words}

\begin{setup}
Throughout \S\ref{subsec-admissible-words}, let $Q$ be a quiver and $Z$ be a set of quadratic zero-relations such that $(Q,Z)$ is locally gentle. See \Cref{def-locally-gentle-pair} for details.
\end{setup}

\begin{defn}\label{defn:letters} 
A \emph{letter} is a symbol of the form $a$  or $a^{-1}$ with $a\in Q_{1}$. 
Those of the form $a$ are \emph{direct letters}, and those of the form $a^{-1}$ are \emph{inverse letters}. 
Each direct letter comes equipped with a head and a tail, namely those of the corresponding arrow. 
The head and tail of an inverse letter are defined by $h(a^{-1})=t(a)$ and $t(a^{-1})=h(a)$ for each $a\in Q_{1}$. 
The \emph{inverse} of a letter refers to the involution $a\leftrightarrow a^{-1}$ of the set of letters. 
\end{defn}

\begin{defn}\label{defn:words}
By a \emph{trivial word} we mean a symbol of the form $1_{v,\epsilon}$ where $v\in Q_{0}$ and $\epsilon=\pm 1$. 
By \emph{non}-\emph{trivial word} $C$  we refer to a  sequence $C=\dots C_{i}\dots$ of letters $C_{i}$, either of the form $C_{1}\dots C_{n}$ (referred to as \emph{finite}) or of the form $\dots C_{-1}C_{0}\vert C_{1}C_{2}\dots$ (referred to as \emph{infinite}), subject to the conditions that for all appropriate $i$ we have that $t(C_{i})=h(C_{i+1})$ and that  $C_{i}^{-1}\neq C_{i+1}$. 
In this situation we write $v_{i}(C)$ for the vertex $t(C_{i})=h(C_{i+1})$ in $Q$. For finite words, $v_0(C)=h(C_{1})$ and $v_n(C)=t(C_{n})$.  
By a \emph{word} we mean either a trivial word, or a non-trivial word that may be finite or infinite. 
\end{defn}

\begin{defn}\label{defn:admissible_word}
We say a word $C$ is \emph{admissible} if, for all consecutive pairs $C_i, C_{i+1}$, $C_{i}C_{i+1}\notin Z$ when $C_{i}$ and $C_{i+1}$ are direct, and $C^{-1}_{i+1}C^{-1}_{i}\notin Z$ when $C_{i}$ and $C_{i+1}$ are inverse. 
\end{defn}

\begin{rem}
    Our notion of an admissible word here is equivalent to the notion of an admissible word from \cite{BenTenCraBoe}. Unlike  \cite{BenTenCraBoe}, in this article there are no ${*}$-letters, and  so every word is \emph{end}-\emph{admissible}. 
Likewise, in this article, the only infinite words are $\BZ$-\emph{indexed}. 
Similarly, since $Z$ is quadratic, a path $p= a_{1}\dots  a_{n}$ is \emph{relation}-\emph{admissible} if and only if it is admissible, meaning that each subpath $ a_{i} a_{i+1}$ of length $2$ lies outside $Z$. 
\end{rem}

\begin{defn}\label{defn:equivalent-words}
We define the \emph{inverse} $C^{-1}$ of a word $C$ as follows:
\begin{description}
    \item[Trivial word] For  $C=1_{v,\epsilon}$ we let  $C^{-1}=1_{v,-\epsilon}$.
    \item[Finite word] For $C=C_{1}\dots C_{n}$ we let  $C^{-1}=C_{n}^{-1}\dots C_{1}^{-1}$
    \item[Infinite word] For $C=\dots C_{-1}C_{0}\vert C_{1}C_{2}\dots$ we let  $C^{-1}=\dots C^{-1}_{1}C^{-1}_{0}\vert C^{-1}_{-1}C^{-1}_{-2}\dots$
\end{description}
\end{defn}

 \begin{defn}\label{defn:shift}
 For any $d\in \BZ$ we define the \emph{shift} $C[d]$ of a word $C$ by $C[d]=C$ when $C$ is trivial or finite, and by $C[d]=\dots C_{d-1}C_{d}\vert C_{d+1}C_{d+2}\dots$ when $C=\dots C_{-1}C_{0}\vert C_{1}C_{2}\dots$ is infinite. 
 \end{defn}
 
\begin{defn}\label{defn:periodic}
We say that an infinite word $C$ is \emph{periodic} if $C[n]=C$ for some $n>0$, and we call the minimal such $n>0$ the  \emph{period} of $C$. 
\end{defn}
 
 Define a relation on the set of all words $C$ and $D$ by declaring $C\sim D$ if and only if there exists $d\in \BZ$ such that $D=C[d]$ or $D=C^{-1}[d]$. 
 Note $\sim$ is an equivalence relation that preserves and reflects the properties of being trivial, finite, infinite and admissible. 

 \begin{example}\label{running-example-words}
 In the gentle pair from Example \ref{running-example-quiver}, examples of admissible finite words include $\nu \zeta^{-1}\sim \zeta\nu^{-1}$ and $\eta\delta^{-1}\alpha\nu\beta\alpha\nu \zeta^{-1} \sim \zeta \nu^{-1} \alpha^{-1} \beta^{-1} \nu^{-1} \alpha^{-1} \delta \eta^{-1}$.

An example of an infinite word is given by infinite repetitions of the sequence  $\nu\beta\alpha$.
     
\end{example}

\subsection{Associated bimodules}
\label{subsec-associated-bimodules}

Let $(Q,Z)$ be locally gentle and let $C$ be an admissible word. We define a quiver from a word $C$ via the following algorithm.

\begin{defn}\label{defn:quiver-from-word} 
For any word $C$, we define a \emph{quiver $Q(C)$}, by firstly defining the set $V(C)$ of vertices and the set $A(C)$ of arrows, as follows. 

\noindent \textbf{Vertices}: We go through the cases for various $C$:
\begin{itemize}
\item If $C$ is trivial, let $V(C)=\{0\}$.
\item If $C=C_{1}\dots C_{n}$ is finite, let $V(C)=\{0,\dots,n\}$.
\item If $C$ is infinite, let $V(C)=\BZ$. 
\end{itemize}
\noindent \textbf{Arrows}: Define $A(C)$ and the head and tail functions $h_{C},t_{C}\colon A(C)\to V(C)$ as follows.\\
\noindent If $i,i-1\in V(C)$ then there is exactly one arrow $\theta_{i}$ incident at both $i$ and $i-1$, with 
\[
\begin{array}{ccc}
h_{C}(\theta_{i})=i-1,&t_{C}(\theta_{i})=i,& (C_{i}\text{ direct}).
\\
h_{C}(\theta_{i})=i,&t_{C}(\theta_{i})=i-1,& (C_{i}\text{ inverse}).
\end{array}
\]
The underlying graph of the quiver $Q(C)$ is $\mathbb{A}_{n}$ when $C$ is finite and ${}_{\infty}\mathbb{A}_{\infty}$ when $C$ is infinite: vertices increase from left to right; and direct (respectively, inverse) letters $C_{i}$ give left-pointing (respectively, right-pointing) arrows $\theta_{i}$ in $Q(C)$.

With $Q(C)$ defined we define a homomorphism of quivers $f_{C}\colon Q(C)\to Q$ on vertices by   $V(C)\ni i\mapsto v_{i}(C)\in Q_{0}$ and on arrows by  $A(C)\ni \theta_{i}\mapsto a\in Q_{1}$ whenever $C_{i}=a^{\pm1}$. 
\end{defn}

\begin{example}\label{running-example-associated-quivers}
Taking  $C=\eta\delta^{-1}\alpha\nu\beta\alpha\nu \zeta^{-1}$ in Example \ref{running-example-quiver}, the quiver $Q(C)$ is \[0\xleftarrow{\theta_{1}}1 \xrightarrow{\theta_{2}}2\xleftarrow{\theta_{3}}3\xleftarrow{\theta_{4}}4\xleftarrow{\theta_{5}}5\xleftarrow{\theta_{6}}6\xleftarrow{\theta_{7}}7\xrightarrow{\theta_{8}}8.\] and its image under $f_{C}$ may be depicted by
  \[6\xleftarrow{\eta}5 \xrightarrow{\delta}2\xleftarrow{\alpha}1\xleftarrow{\nu}3\xleftarrow{\beta}2\xleftarrow{\alpha}1\xleftarrow{\nu}3\xrightarrow{\zeta}4.\]
 \end{example}

\begin{defn}\label{defn:natural-automorphisms} 
\cite[Definition 2.20]{BenTenCraBoe} 
Let $\boldsymbol{\sigma}\colon Q_{1}\to \Aut(K)$ be a function. Let $\pi_{0}$ be the identity on $K$. Suppose $i,j\in V(C)$ and $\vert j-i\vert =1$, so that there is precisely one arrow $\theta\in A(C)$ incident at $i$ and $j$, say where $f_{C}(\theta)=a$. If $\pi_{t(\theta)}$ is defined, then let $\pi_{h(\theta)}=\sigma_{a}\pi_{t(\theta)}$, and if $\pi_{h(\theta)}$ is defined, let $\pi_{t(\theta)}=\sigma^{-1}_{a}\pi_{h(\theta)}$. 
Hence iterating defines $\pi_{i}$ for all $i\in V(C)$. Note that $\pi_{h(\theta)}=\sigma_{f_{C}(\theta)}\pi_{t(\theta)}$ for each $\theta\in A(C)$. 
\end{defn}

\begin{example}\label{ex:pi}
Let $C = \nu \zeta^{-1}$. Then, $\pi_0 = id_K, \pi_1 = \sigma_\nu^{-1}$, and $\pi_2 = \sigma_\zeta \sigma_\nu^{-1}$.

Now let $C = \eta\delta^{-1}\alpha\nu\beta\alpha\nu \zeta^{-1}$. The first few values of $\pi_i$ for this word are as follows: $\pi_0 = id_K, \pi_1 = \sigma_\eta^{-1}, \pi_2 = \sigma_\delta \sigma_{\eta}^{-1}, \pi_3 = \sigma_\alpha^{-1}\sigma_\delta\sigma_{\eta}^{-1}$, and $\pi_4 = \sigma_\nu^{-1}\sigma_\alpha^{-1}\sigma_\delta\sigma_{\eta}^{-1}$. 
\end{example}

We now define the left and right module structures.

\begin{defn}\label{defn:left-module-structure}
\cite[Definition 2.17]{BenTenCraBoe}
Let $\Lambda=K_{\boldsymbol{\sigma}}Q/\langle Z\rangle$. 
Define the \emph{left $\Lambda$-module} $M(C)=F(C)/T(C)$ where $F(C)=\bigoplus_{i\in V(C)}\Lambda b_{i}$, the free left $\Lambda$-module with basis $(b_{i}\colon i\in V(C))$, and $T(C)$ is the submodule of $F(C)$ generated by
\[
\begin{array}{cc}
e_{f_{C}(i)}b_{i}-b_{i}, & (i\in V(C)), \\
f_{C}(\theta)b_{t(\theta)}-b_{h(\theta)}, & (\theta\in A(C)),\\
 a b_{j}, & (j\in V(C),\,a\in Q_{1}\setminus  \{f_{C}(\theta)\colon \theta\in A(C),\,t_{C}(\theta)=j\}).
\end{array} 
\]
\end{defn}


\begin{rem} 
Note that if $i\neq f_{C}(j)$ for all $j\in V(C)$ then $e_{i}b_{j}=-e_{i}(e_{f_{C}(j)}b_{j}-b_{j})+e_{i}e_{f_{C}(j)}b_{j}$, but also  $e_{i}e_{f_{C}(j)}=0$ inside $\Lambda$, and so $e_{i}b_{j}\in T(C)$.
By \cite[Lemma 2.19]{BenTenCraBoe} the elements $b_{i}$ with $i\in V(C)$  form a left $K$-basis of the module $M(C)$ if and only if the word $C$ is admissible in the sense of \Cref{defn:words}. 
\end{rem}

\begin{defn}\label{defn:right-module-structure}
Recall the function $\pi\colon V(C)\to \Aut(K)$ defined in \Cref{defn:natural-automorphisms}. By \cite[Lemma 2.21]{BenTenCraBoe}, there is a unique $\Lambda$-$K$-bimodule action on $M(C)$ such that $b_{i}\lambda=\pi_{i}(\lambda)b_{i}$ for each $\lambda\in K$. If $C$ is infinite and periodic of period $n>0$, then we let $\pi_{C}=\pi_{n}^{-1}$. Then, by \cite[Lemma 2.26]{BenTenCraBoe}, there is a unique way to extend the \emph{right $K$-module structure} to make $M(C)$ into a $\Lambda$-$K[t,t^{-1};\pi_{C}]$-bimodule such that $b_{i}t=b_{i-n}$ for each $i\in V(C)$. Furthermore, in this case, as a right $K[t,t^{-1};\pi_{C}]$-module $M(C)$ is free with basis $b_{0},\dots,b_{n-1}$. 
\end{defn}

\subsection{Strings, bands, and parameterising rings}
\label{subsec-string-and-band-modules}

If $C$ is a string, meaning $C$ is finite as an admissible word, then let $\Lambda(C)=K$. 
If $C$ is a band, meaning $C$ is infinite and periodic, then let $\Lambda(C)=K[x,x^{-1};\pi_{C}]$. 
When we say that a module over a $K$-ring is finite-dimensional, we will always mean that the module is finitely generated as a $K$-module.  
If $C$ is a string then, of course, the only finite-dimensional indecomposable $\Lambda(C)$-module is $K$. 
If $C$ is a band, then $\Lambda(C)$ is a (possibly non-commutative) principal ideal domain, and hence the finite-dimensional indecomposable $\Lambda(C)$-modules have been classified  by Jacobson \cite[Chapter 3, \S 8]{Jacobson-theory-of-rings}. 
We now recall the main result from \cite{BenTenCraBoe}. 

\begin{thm}
\label{BTCB-main-theorem}
    \cite[Main Theorem, p. 2]{BenTenCraBoe}
    Let $\Lambda$ be a semilinear locally gentle algebra. 
    As $C$ runs through representatives of the equivalence classes of strings and bands, and as $V$ runs through a complete set of pairwise non-isomorphic finite-dimensional indecomposable $\Lambda(C)$-modules, the modules $M(C)\otimes_{\Lambda(C)}V$ run through a complete list of pairwise non-isomorphic finite-dimensional indecomposable $\Lambda$-modules. 
\end{thm}

\subsection{Geometric model}\label{subsec:our-model}

\begin{setup}\label{setup:our-model}
In Section \ref{subsec:our-model}, let  $\overline{Q} = (Q,Z)$ be a locally gentle pair. Let $(S_{\overline{Q}}, V_{\overline{Q}} \cup V^*_{\overline{Q}})$ be a surface with marked points and with pair of dual cellular dissections $D_{\overline{Q}}, D^*_{\overline{Q}}$ associated to $\overline{Q}$, as in \S\ref{subsec-surfaces-and-locally-gentle-pairs}. Now, for every bigon $F \in \mathcal{F}(D)$ with one edge on the boundary of $S_{\overline{Q}}$, we include a new marked point in $V_{\overline{Q}}$ on this boundary segment. We extend the notion of a labeled tiling to include these additional marked points.

We also let $\Lambda = K_{\boldsymbol{\sigma}}Q/I$ for a division ring $K$, a  map $\boldsymbol{\sigma}\colon Q_1 \to \Aut(K)$ and $I = \langle Z \rangle$. 
\end{setup}

In Theorem \ref{BTCB-main-theorem}, we saw that the strings and bands of a gentle pair determined the indecomposable modules of a corresponding semilinear algebra. Here, we connect this story with arcs and closed curves on dissected surfaces. We begin by comparing arcs and closed curves with words. We extend to the locally gentle case here. The correspondence between strings and arcs was already shown by Palu, Pilaud, and Plamondon \cite{Palu-Pilaud-Plamondon-non-kissing-non-crossing}.

\begin{prop}[Proposition 4.23 \cite{Palu-Pilaud-Plamondon-non-kissing-non-crossing}]\label{prop:bijectionPPP}
There is a bijection between equivalence classes of strings (i.e. finite words) of $\overline{Q}$ and equivalence classes of permissible arcs on the surface with dissection $(S_{\overline{Q}},V_{\overline{Q}} \cup V^*_{\overline{Q}},D_{\overline{Q}})$.
\end{prop}

We note that the same statement was shown in the gentle case in Theorem 3.8 of \cite{Baur-Coelho-Simoes-geometric-model-module-cat-gentle}. The correspondence from arcs to strings works by recording the labels of the faces of type 2 which an arc passes through when it crosses both bounding arcs from $D$. These faces are labeled with elements of $Q_1$, and the sign these have in the corresponding word is determined by whether the shared endpoint of these two arcs from $D$ which is on the boundary of this face lies to the left or right of $\gamma$ with its given orientation. As we will see in the proof of Proposition \ref{prop:BijArcsStrings}, with our conventions, we have a direct letter exactly when this endpoint lies to the right of $\gamma$. 

\begin{rem}
    There is a difference between the considerations made in this article and those made by Palu, Pilaud and Plamondon in \cite{Palu-Pilaud-Plamondon-non-kissing-non-crossing}. 
    On the one hand, any string module $M(C)$ we consider here is finite-dimensional over  $K$ since the word $C$ used to define it is finite. 
    On the other hand, the words used to define string modules  in \cite{Palu-Pilaud-Plamondon-non-kissing-non-crossing} may be infinite (but must eventually be periodic), and correspond to what they call \emph{accordions}; see \cite[Definition 3.8]{Palu-Pilaud-Plamondon-non-kissing-non-crossing}. 
    Note that both in this article and in \cite{Palu-Pilaud-Plamondon-non-kissing-non-crossing}, doubly-infinite periodic words are used to define band modules.  
\end{rem}

We show a similar statement relating bands and closed curves on these surfaces, thus generalizing Proposition 3.9 in \cite{Baur-Coelho-Simoes-geometric-model-module-cat-gentle}. Recall from Definition \ref{def:PermissibleArc} that a closed curve which cuts out a disk with $m \geq 0$ punctures from $V^*$ and no punctures from $V$ is not permissible. 

\begin{prop}\label{prop:BijArcsStrings}
There is a bijection between equivalence classes of bands (i.e. doubly-infinite periodic words) of $\overline{Q}$ and repetition equivalence classes of permissible closed curves on the surface with dissection $(S_{\overline{Q}},V_{\overline{Q}} \cup V^*_{\overline{Q}},D_{\overline{Q}})$.
\end{prop}

\begin{proof}
This proof relies on similar techniques as for the gentle case, shown in \cite{Baur-Coelho-Simoes-geometric-model-module-cat-gentle}. 

Given a repetition equivalence class of a permissible closed curve, we choose a primitive representative $\delta: [0,1] \to S$ such that $\delta(0) = \delta(1)$ does not lie on an arc in $D$ and $\delta$ crosses a minimal number of arcs from $D$. Suppose that $\delta$ crosses arcs $\rho_1,\ldots,\rho_d$ and passes through faces $F_0,\ldots,F_d \in \mathcal{F}(D_{\overline{Q}} \cup R_{\overline{Q}}^*)$, with both lists ordered based on the orientation of $\delta$. Our choice of representative guarantees that all faces $F_i$ are of type 2, as in Definition \ref{def:SemilinearTiling} so that they have two bounding arcs from $D_{\overline{Q}}$. Necessarily, $F_0 = F_d$, $\rho_1$ and $\rho_d$ border $F_0 = F_d$, and for $j < d$, $\rho_j$ and $\rho_{j+1}$ border $F_j$.  
 By Corollary \ref{cor:ArrowsToFaces}, the map $\mathcal{L}$ associates an arrow to each such face such that, if $\tau_j$ and $\tau_k$ border $F$, then $F$ is labeled with an arrow $a$ such that $\{h(a),t(a)\} = \{\tau_j, \tau_k\}$. Let $\mathcal{L}(F_j) = a_j$. 

Since $\delta$ is permissible, we know that for each $j < d$, the arcs $\rho_j$ and $\rho_{j+1}$ share an endpoint $p_j$ which lies on the boundary of $F_j$ and such that the connected component of $F_j \backslash \delta$ which contains $p_j$ is an unpunctured disk. Similarly, $\tau_1$ and $\tau_d$ must share an endpoint, which we call $p_d$. For $1 \leq j \leq d$, if $p_j$ lies to the right of $\delta$, let $C_j = a_j$; otherwise, let $C_j = a_j^{-1}$. Necessarily $t(C_d) =  h(C_1)$, and we know that for all $j \leq d$, $C_{j} C_{j+1}$ and $C_{j+1}^{-1} C_{j}^{-1}$ are not in $Z$ by the permissibility of $\delta$ where $C_{d+1}:= C_1$.  
Hence setting $C=C_1 \cdots C_d$, we have that ${}^{\infty}C^{\infty}$ is an admissible, doubly-infinite  periodic word. 

This process can be reversed. That is, given an admissible, doubly-infinite periodic word of period $d$, we take the defining subword $C_1 \cdots C_d$. For each $1 \leq i \leq d$ we draw an oriented segment from $\tau_{t(C_i)}$ to $\tau_{h(C_i)}$. We draw these segments such that property (4) in Definition \ref{def:PermissibleArc} is satisfied. Since $t(C_i) = h(C_{i+1})$ and since $C[d] = C$, $t(C_d) = h(C_1)$, we can connect these segments to form a closed curve. Property (3) of Definition \ref{def:PermissibleArc} is guaranteed since we began with an admissible word. Our choices of arc segments ensures that $\delta$ does not wind around punctures from $V^*$. The admissibility of the word ensures that we cannot find an closed curve homotopic to $\delta$ which is contractible or encloses a disk with punctures only from $V^*$. Therefore, we have built a permissible closed curve.

From these two constructions, it is clear that choosing other primitive representatives of a repetition equivalence class of $\delta$ will generate an equivalent word and vice versa. 
\end{proof}

\begin{cor}\label{cor:BijArcsModules}
There is a correspondence between isomorphism classes of indecomposable finite-dimensional modules over  $\Lambda$ and equivalence classes of permissible arcs or closed curves of $(S_{\overline{Q}},V_{\overline{Q}} \cup V^*_{\overline{Q}},D_{\overline{Q}})$. Explicitly, equivalence classes of permissible arcs correspond to string modules, while repetition equivalence classes (see Definition \ref{def:EquivClosedCurves}) of  permissible closed curves correspond to band modules. 
\end{cor}

\begin{proof}
    Combine \cref{prop:bijectionPPP} and  \cref{prop:BijArcsStrings} with \cref{BTCB-main-theorem}. 
\end{proof}

Definitions \ref{defn-semilinearity-of-permissible-arcs} and \ref{defn-semilinearity-of-closed-curves} introduce the \emph{semilinearity} of  
 arcs and closed curves. 
\begin{defn}
\label{defn-semilinearity-of-permissible-arcs}
    Let $(S_{\overline{Q}},V_{\overline{Q}} \cup V^*_{\overline{Q}}, D_{\overline{Q}}, R^*_{\overline{Q}}, \boldsymbol{\ell})$ be a labeled tiling coming from $\overline{Q}$, and let $\gamma\colon [0,1]\to S$ be a permissible arc.
    Write  $(\rho_{0},\rho_{1},\dots,\rho_{d-1})$ for the sequence of arcs in $D_{\overline{Q}}$ that $\gamma$ crosses. 
    Since $\gamma$ is a permissible arc, for each $0 \leq i < d-1$, the arcs $\rho_i$ and $\rho_{i+1}$ share an endpoint $p_i$ such that the disk cut out by $\gamma,\rho_i,$ and $\rho_{i+1}$ which contains $p_i$ does not contain any point from $V^*$. Let also $F_i \in \mathcal{F}(D \cup R^*)$ be the face which $\gamma$ passes through immediately before crossing $\rho_{i+1}$, where $F_d$ is the last face $\gamma$ passes through. Note that $F_i, 1 \leq i \leq d-1$ is necessarily a face of type 2.
    
    For each $i=0,\dots,d-1$ we define the \emph{semilinearity of} $\gamma$ \emph{in place} $i$ to be the automorphism $\sigma_{\gamma,i}\in \Aut(K)$ defined as follows.     
    Firstly let   $\sigma_{\gamma,0}$ be the identity on $K$. 
    Secondly, for $0\leq i< d-1$, if $p_i$ lies to the right of $\gamma$, then $\sigma_{\gamma,i+1}=\boldsymbol{\ell}(F_i)^{-1}\circ \sigma_{\gamma,i}$; otherwise, $p_i$ lies to the left of $\gamma$, and $\sigma_{\gamma,i+1}=\boldsymbol{\ell}(F_i)\circ \sigma_{\gamma,i}$.
\end{defn}


\begin{defn}
\label{defn-semilinearity-of-closed-curves}
     Let $(S_{\overline{Q}},V_{\overline{Q}} \cup V^*_{\overline{Q}}, D_{\overline{Q}}, R^*_{\overline{Q}}, \boldsymbol{\ell})$ be a labeled tiling coming from $\overline{Q}$, and let $\gamma\colon [0,1]\to S$ be a closed curve. 
    Write   $(\rho_{0},\dots,\rho_{d-1})$ for the sequence of arcs in $D_{\overline{Q}}$ that $\gamma$ crosses . 
   Since $\gamma$ is a permissible closed curve, for each $0 \leq i < d$, the arcs $\rho_i$ and $\rho_{i+1}$ share an endpoint $p_i$ such that the disk cut out by $\gamma,\rho_i,$ and $\rho_{i+1}$ which contains $p_i$ does not contain any point from $V^*$. The same is true for $\rho_0$ and $\rho_{d-1}$, and we call this point $p_{d-1}$. Let also $F_i \in \mathcal{F}(D \cup R^*)$ be the face which $\gamma$ passes through immediately before crossing $\rho_{i+1}$, and set $F_d = F_0$.  Note that all faces $F_i$ are of type 2.  
    
    For each $i=0,\dots,d$ we define the \emph{semilinearity of} $\gamma$ \emph{in place} $i$ to be the automorphism $\sigma_{\gamma,i}\in \Aut(K)$ defined as follows.     
    Firstly let $\sigma_{\gamma,0}$ be the identity on $K$. 
    Secondly, for $0\leq i< d-1$, if $p_i$ lies to the right of $\gamma$, then $\sigma_{\gamma,i+1}=\boldsymbol{\ell}(F_i)^{-1}\circ \sigma_{\gamma,i}$; otherwise, $p_i$ lies to the left of $\gamma$, and $\sigma_{\gamma,i+1}=\boldsymbol{\ell}(F_i)\circ \sigma_{\gamma,i}$.
   Thirdly, if $p_{d-1}$ lies to the right of $\gamma$, we let $\sigma_{\gamma,d}=\boldsymbol{\ell}(F_{d-1})^{-1}\circ \sigma_{\gamma,p-1}$; otherwise, $p_{d-1}$ lies to the left of $\gamma$, and $\sigma_{\gamma,d} = \boldsymbol{\ell}(F_{d-1})\circ \sigma_{\gamma,d-1}$.

\end{defn}

Our final main result is showing that arcs and closed curves on labeled tilings model the set of finite-dimensional indecomposable modules of a semilinear locally gentle algebras. 

\begin{thm}\label{thm:PiMatchesArcSemilinearity}
Let $(S_{\overline{Q}}, V_{\overline{Q}} \cup V^*_{\overline{Q}}, D_{\overline{Q}}, R^*_{\overline{Q}}, \boldsymbol{\ell})$ be the labeled tiling associated to $\Lambda$. Let $M=M(C)\otimes_{\Lambda(C)} V$ for a string or band $C$ and $V$ a finite-dimensional indecomposable $\Lambda(C)$-module. 
Let $\gamma$ be a permissible arc or closed curve in $S$ corresponding to $M$. 

\begin{enumerate}
    \item For any $i\in V(C)$ the automorphism $\pi_{i}$ from \Cref{defn:natural-automorphisms} is precisely the semilinearity of the permissible arc $\gamma$ at $i$. 
    \item If $C$ is doubly-infinite periodic of period $n$, then the automorphism  $\pi^{-1}_C$ from \Cref{defn:right-module-structure} is precisely the semilinearity of the closed curve $\gamma$ at $n$. 
\end{enumerate}
\end{thm}

\begin{proof}
Recall \Cref{cor:BijArcsModules}. 
Given a string or band $C$, we prove $\pi_i = \sigma_{\gamma,i}$ inductively. This is clear for $i = 0$, as $\pi_i = id_K = \sigma_{\gamma,i}$. Now assume that $C$ is either doubly-periodic infinite or finite with length $d-1 \geq i$, and assume $\pi_{i-1} = \sigma_{\gamma,i-1}$. Assume first that $C_i = a$ is direct. This means $i-1 = h(\theta_i)$, so by \Cref{defn:natural-automorphisms}, $\pi_i = \sigma_{a}^{-1} \circ \pi_{i-1}$. 

In the bijection described in Propositions \ref{prop:bijectionPPP} and \ref{prop:BijArcsStrings}, we know that the arc $\gamma$ associated to $C$ crosses $m \geq d-1$ arcs, $\rho_0,\ldots,\rho_{m} \in D_{\overline{Q}}$ where for $0 \leq j \leq m$, $\rho_j = \tau_{v_j(C)}$. Moreover, between crossing $\rho_{j-1}$ and $\rho_j$, $\gamma$ passes through the face $F_j \in \mathcal{F}(D_{\overline{Q} \cup R_{\overline{Q}}^*})$ such that $\mathcal{L}^{-1}(F_j) = f_C(\theta_j)$. So in particular, between the intersection points of $\rho_{i-1}$ and $\rho_i$, $\gamma$ passes through a face $F_i$ with $\mathcal{L}^{-1}(F_i) = a$. The fact that $C_i = a$ is direct means that, in $F$, $\rho_{i-1}$ and $\rho_i$ share an endpoint which lies to the right of $\gamma$. Therefore, $\sigma_{\gamma,i} = \boldsymbol{\ell}(F)^{-1} \circ  \sigma_{\gamma,i} = \sigma_{\mathcal{L}^{-1}(F)}^{-1} \circ  \sigma_{\gamma,i-1}  = \sigma_a^{-1} \circ  \sigma_{\gamma,i-1}$, implying that $\pi_i = \sigma_{\gamma,i}$.

If instead $C_i$ was an inverse arrow, we would have a similar process to show that  knowing $\pi_{i-1} = \sigma_{\gamma,i-1}$ implies $\pi_i = \sigma_{\gamma,i}$. 
Hence (1) holds, and (2) follows immediately. 
\end{proof}

\begin{example}\label{running-example:model}
We once again use the algebra from Example \ref{running-example:zembyk-excision}, whose surface model is given in Example \ref{running-example:tiling}. We add two marked points to $V_{Q}$, as per Setup \ref{setup:our-model}. 

\Cref{subfig:running-example:easy} shows the permissible arc $g$ corresponding to the string $\zeta\nu^{-1} \sim \nu\zeta^{-1}$. 
If we follow $g$ going from the top to the bottom, the arcs in $D_{\overline{Q}}$ that we cross are $(\rho_0, \rho_1, \rho_2)=(\tau_1, \tau_3, \tau_4)$, in that order (using the correspondence between arcs in $D_{\overline{Q}}$ and vertices of $Q$). As per \Cref{defn-semilinearity-of-permissible-arcs}, the semilinearity of $g$ at place $i = 0,1,2$ is 
given by $   \sigma_{g, 0} = 1_K$, $\sigma_{g, 1} = \sigma_\nu^{-1}$ and $\sigma_{g, 2} = \sigma_\zeta  \sigma_\nu^{-1}$. 
Comparing this with Example \ref{ex:pi} demonstrates how Theorem \ref{thm:PiMatchesArcSemilinearity} works.

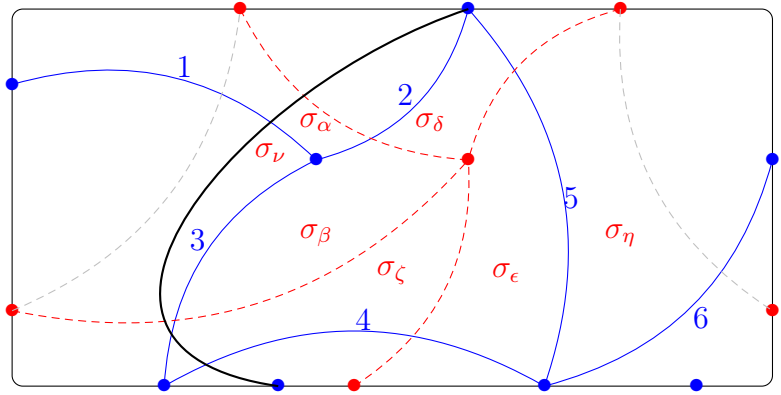
\begin{figure}[H]
    \centering
    \begin{tikzpicture}[inner sep=0.5, outer sep = 0]
        \draw[rounded corners] (0,1)--(0,5)--(10,5)--(10,0)--(0,0)--(0,1);
        
        \begin{scope}[blue]
        \node  (a) at (0,4) {$\bullet$};
        \node (b) at (2,0) {$\bullet$};
        \node (c) at (4,3) {$\bullet$};
        \node (d) at (6,5) {$\bullet$};
        \node (e) at (7,0) {$\bullet$};
        \node (f) at (10,3) {$\bullet$};
        \node (g) at (9,0) {$\bullet$};
        \node (h) at (3.5,0) {$\bullet$};

        \draw (a.center) to[bend left] node[above right] {1} (c.center);
        \draw (d.center) to[bend left] node[above left] {2} (c.center);
        \draw (b.center) to[bend left] node[above left] {3} (c.center);
        \draw (b.center) to[bend left] node[above right] {4} (e.center);
        \draw (d.center) to[bend left] node[below right] {5} (e.center);
        \draw (f.center) to[bend left] node[below right] {6} (e.center);

        \end{scope}
        \begin{scope}[red, densely dashed]
        \node (i) at (3,5) {$\bullet$};
        \node (ii) at (0,1) {$\bullet$};
        \node (iii) at (6,3) {$\bullet$};
        \node (iv) at (4.5,0) {$\bullet$};
        \node (v) at (8,5) {$\bullet$};
        \node (vi) at (10,1) {$\bullet$};
        
        \draw (i.center) to[bend right] (iii.center);
        \draw (ii.center) to[bend right] (iii.center);
        \draw (iv.center) to[bend right] (iii.center);
        \draw (v.center) to[bend right] (iii.center);
        
        \node at (4, 3.5) {$\sigma_\alpha$};
        \node at (4, 2) {$\sigma_\beta$};
        \node at (3.4, 3.1) {$\sigma_\nu$};
        \node at (5.5, 3.5) {$\sigma_\delta$};
        \node at (6.5, 1.5) {$\sigma_\epsilon$};
        \node at (8, 2) {$\sigma_\eta$};
        \node at (5, 1.5) {$\sigma_\zeta$};
        \end{scope}
         
        \begin{scope}[lightgray, densely dashed]
        \draw (ii.center) to[bend right] (i.center);
        \draw (v.center) to[bend right] (vi.center);
        \end{scope}
        
        \begin{scope}[black, thick]
        \draw (h.center) .. controls (0.05,0.5) and (3,4) 
        ..(d.center);
        \end{scope}
        
    \end{tikzpicture}
    \caption{A permissible arc corresponding to a string module. }
    \label{subfig:running-example:easy}
\end{figure}

\Cref{subfig:running-example:complicated} shows the permissible arc corresponding to the string $\eta\delta^{-1}\alpha\nu\beta\alpha\nu \zeta^{-1}$. One can check that the semilinearity of this arc matches the values $\pi_i$ given in Example \ref{ex:pi}.

\begin{figure}[H]
    \centering
    \begin{tikzpicture}[inner sep=0.5, outer sep = 0]
        \draw[rounded corners] (0,1)--(0,5)--(10,5)--(10,0)--(0,0)--(0,1);
        
        \begin{scope}[blue]
        \node  (a) at (0,4) {$\bullet$};
        \node (b) at (2,0) {$\bullet$};
        \node (c) at (4,3) {$\bullet$};
        \node (d) at (6,5) {$\bullet$};
        \node (e) at (7,0) {$\bullet$};
        \node (f) at (10,3) {$\bullet$};
        \node (g) at (9,0) {$\bullet$};
        \node (h) at (3.5,0) {$\bullet$};

        \draw (a.center) to[bend left] node[above right] {1} (c.center);
        \draw (d.center) to[bend left] node[above left] {2} (c.center);
        \draw (b.center) to[bend left] node[above left] {3} (c.center);
        \draw (b.center) to[bend left] node[above right] {4} (e.center);
        \draw (d.center) to[bend left] node[below right] {5} (e.center);
        \draw (f.center) to[bend left] node[below right] {6} (e.center);

        \end{scope}
        \begin{scope}[red, densely dashed]
        \node (i) at (3,5) {$\bullet$};
        \node (ii) at (0,1) {$\bullet$};
        \node (iii) at (6,3) {$\bullet$};
        \node (iv) at (4.5,0) {$\bullet$};
        \node (v) at (8,5) {$\bullet$};
        \node (vi) at (10,1) {$\bullet$};
        
        \draw (i.center) to[bend right] (iii.center);
        \draw (ii.center) to[bend right] (iii.center);
        \draw (iv.center) to[bend right] (iii.center);
        \draw (v.center) to[bend right] (iii.center);
        
        \node at (3.2, 4) {$\sigma_\alpha$};
        \node at (4, 2.5) {$\sigma_\beta$};
        \node at (3.4, 3) {$\sigma_\nu$};
        \node at (5.5, 3.5) {$\sigma_\delta$};
        \node at (6.5, 1.5) {$\sigma_\epsilon$};
        \node at (8, 2) {$\sigma_\eta$};
        \node at (5, 1.5) {$\sigma_\zeta$};
                \end{scope}
         
        \begin{scope}[lightgray, densely dashed]
        \draw (ii.center) to[bend right] (i.center);
        \draw (v.center) to[bend right] (vi.center);
        \end{scope}

        \begin{scope}[black, thick]
        \draw (g.center) .. controls (8,6) and (6,4) .. (4,4) arc[start angle=90, end angle=270, radius=1] -- (3.9,2) arc[start angle= 270, delta angle = 180, radius=0.7]  --  (3.6,3.35) .. controls (2.5,3) and (0.25,1) ..  (h.center);
        \end{scope}

    \end{tikzpicture}
    \caption{A permissible arc that winds around a marked point. }
    \label{subfig:running-example:complicated}

\end{figure}
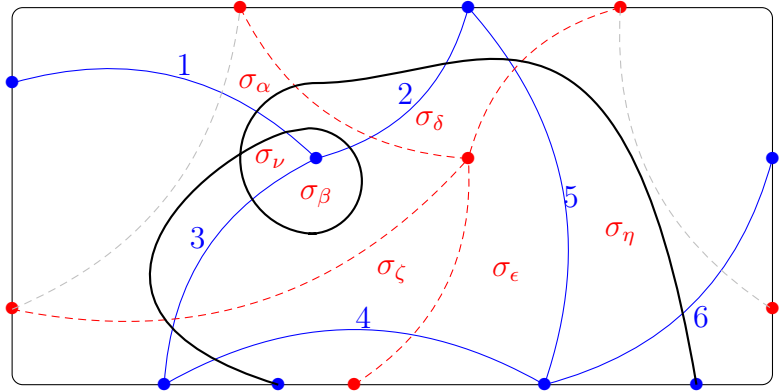

\Cref{subfig:running-example:closed} shows the closed curve $\mu$ corresponding to the periodic band $\ldots \nu\beta\alpha\ldots$. Suppose that $\mu(0)$ lies in the face labeled by $\sigma_\nu$ and the parameterisation is such that $\mu$ travels clockwise. Then, following \Cref{defn-semilinearity-of-closed-curves}, we can compute that $\sigma_{\mu,3} = \sigma_\nu^{-1} \sigma_\beta^{-1} \sigma_\alpha^{-1}$.

\begin{figure}[H]
    \centering
    \begin{tikzpicture}[inner sep=0.5, outer sep = 0]
        \draw[rounded corners] (0,1)--(0,5)--(10,5)--(10,0)--(0,0)--(0,1);
        
        \begin{scope}[blue]
        \node  (a) at (0,4) {$\bullet$};
        \node (b) at (2,0) {$\bullet$};
        \node (c) at (3.3,3) {$\bullet$};
        \node (d) at (6,5) {$\bullet$};
        \node (e) at (7,0) {$\bullet$};
        \node (f) at (10,3) {$\bullet$};
        \node (g) at (9,0) {$\bullet$};
        \node (h) at (3.5,0) {$\bullet$};

        \draw (a.center) to[bend left] node[above right] {1} (c.center);
        \draw (d.center) to[bend left] node[pos=0.3, above left] {2} (c.center);        
        \draw (b.center) to[bend left] node[above left] {3} (c.center);
        \draw (b.center) to[bend left] node[above right] {4} (e.center);
        \draw (d.center) to[bend left] node[below right] {5} (e.center);
        \draw (f.center) to[bend left] node[below right] {6} (e.center);
        \end{scope}
        
        \begin{scope}[red, densely dashed]
        \node (i) at (3,5) {$\bullet$};
        \node (ii) at (0,1) {$\bullet$};
        \node (iii) at (6,3) {$\bullet$};
        \node (iv) at (4.5,0) {$\bullet$};
        \node (v) at (8,5) {$\bullet$};
        \node (vi) at (10,1) {$\bullet$};
        
        \draw (i.center) to[bend right] (iii.center);
        \draw (ii.center) to[bend right] (iii.center);
        \draw (iv.center) to[bend right] (iii.center);
        \draw (v.center) to[bend right] (iii.center);
        
        \node at (3.5, 3.4) {$\sigma_\alpha$};
        \node at (3.5, 2.5) {$\sigma_\beta$};
        \node at (2.8, 3.1) {$\sigma_\nu$};
        \node at (6, 4) {$\sigma_\delta$};
        \node at (6.5, 1.5) {$\sigma_\epsilon$};
        \node at (8, 2) {$\sigma_\eta$};
        \node at (5, 1.5) {$\sigma_\zeta$};
        \end{scope}
         
        \begin{scope}[lightgray, densely dashed]
        \draw (ii.center) to[bend right] (i.center);
        \draw (v.center) to[bend right] (vi.center);
        \end{scope}
        
        \begin{scope}[black, thick]
        \draw (c) circle[radius=0.9];
        \end{scope}
    \end{tikzpicture}
    \caption{A closed curve.}
    \label{subfig:running-example:closed}
\end{figure}
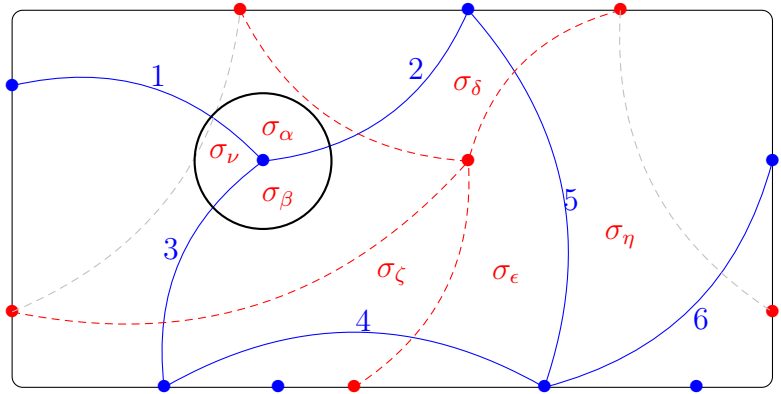

\end{example}

\begin{acknowledgements}
    The authors would like to thank Isobel Webster for useful discussions in the early stages of the development of the article. The first, second and third listed authors are grateful for  support by the Danish National Research Foundation (DNRF156); the Independent Research Fund Denmark (1026-00050B); and the Aarhus University Research Foundation (AUFF-F-2020-7-16). The third listed author was also supported by the Norwegian Research Council via the project "Higher homological algebra and tilting theory" (301046). The fourth listed author thanks Aarhus University for the hospitality during her visit in October 2023.
\end{acknowledgements}

\bibliographystyle{abbrv}
\bibliography{references}


\end{document}